\documentclass[journal]{IEEEtran}

\usepackage{color}

\usepackage{amssymb}
\usepackage{comment}

\usepackage{bbm}
\usepackage{mathrsfs}
\usepackage{verbatim}
\usepackage{amsmath}
\usepackage{amsfonts}
\usepackage{bm}
\usepackage{graphicx}
\usepackage{caption}
\usepackage{float}

\usepackage{amsthm}
\usepackage{xcolor}
\usepackage[colorinlistoftodos]{todonotes}
\usepackage{mathtools}
\usepackage{cite}
\usepackage{bigints}
\usepackage{url}
\makeatletter
\let\NAT@parse\undefined
\makeatother
\usepackage{hyperref}
\hypersetup{
   colorlinks=true,
    linkcolor= blue,
    allcolors=blue,
    citecolor = blue,
    filecolor=black,      
    urlcolor=blue,
    }
    
\usepackage{subfig}

\usepackage{accents}

\theoremstyle{definition}
\newtheorem{assumption}{Assumption}
\newtheorem{definition}{Definition}
\newtheorem{remark}{Remark}
\newtheorem{problem}{Problem}

\theoremstyle{plain}
\newtheorem{proposition}{Proposition}

\newtheorem{theorem}{Theorem}
\newtheorem{lemma}{Lemma}
\newtheorem{corollary}{Corollary}

\newcommand{\eat}[1]{}
\usepackage{csquotes}
\MakeOuterQuote{"}
\usepackage{booktabs}
\usepackage{array}   

\makeatletter

\newcommand{\Rmnum}[1]{\expandafter\@slowromancap\romannumeral #1@}
\allowdisplaybreaks[0]

\begin{document}
              


\title{Time-Optimal Coordination for Connected and Automated Vehicles at Adjacent Intersections}

\author{Behdad Chalaki, {\itshape{Student Member, IEEE}} and Andreas A. Malikopoulos, {\itshape{Senior Member, IEEE}}
\thanks{This research was supported in part by ARPAE's NEXTCAR program under the award number DE-AR0000796 and by the Delaware Energy Institute (DEI).}%
\thanks{The authors are with the Department of Mechanical Engineering, University of Delaware, Newark, DE 19716 USA (email: \texttt{bchalaki@udel.edu; andreas@udel.edu).}} }

\markboth{IEEE Transactions on Intelligent Transportation Systems}{Chalaki and Malikopoulos : TITLE OF THE PAPER}

\maketitle

\begin{abstract}     
In this paper, we provide a hierarchical coordination framework for connected and automated vehicles (CAVs) at two adjacent intersections. This framework consists of an upper-level scheduling problem and a low-level optimal control problem. By partitioning the area around two adjacent intersections into different zones, we formulate a scheduling problem for each individual CAV aimed at minimizing its total travel time. For each CAV, the solution of the upper-level problem designates the arrival times at each zones on its path which becomes the inputs of the low-level problem. The solution of the low-level problem yields the optimal control input (acceleration/deceleration) of each CAV to exit the intersections at the time specified in the upper-level scheduling problem. We validate the performance of our proposed hierarchical framework through extensive numerical simulations and comparison with signalized intersections, centralized scheduling, and FIFO queuing policy. 
\end{abstract}
\indent


\IEEEpeerreviewmaketitle

\section{Introduction}

\subsection{Motivation}

\IEEEPARstart{D}{ue} to the increasing population and travel demand, traffic congestion has become a significant concern in large metropolitan areas. By 2050, it is expected that \(68\%\) of the population will reside in urban areas; by 2030, there would be 41 Mega-cities (with more than 10M people or more) \cite{united20182018}. In $2017$, congestion in urban areas in the US caused drivers to spend an extra $8.8$ billion hours on the road and purchase an extra $3.3$ billion gallons of fuel resulting in $\$166$ billion cost \cite{Schrank2019}. 
In addition, about 35K people in the US lose their lives in traffic accidents each year \cite{USDOT2}.

\subsection{Related Work}
Equipped with computing capabilities and advanced communication technologies, connected and automated vehicles (CAVs) are expected to provide novel and innovative opportunities for users to make better operational decisions to improve both traffic throughput and passenger safety \cite{Klein2016a,Melo2017a,zhao2019enhanced}.
The potential improvement in safety and efficiency of the transportation network by employing a fleet of CAVs can be realized using two main approaches.
The first approach, which gained momentum in the 1980s and 1990s, uses CAVs to reduce vehicle gaps and form high-density platoons to cut congestion \cite{Shladover1991,Rajamani2000}. The second approach smooths the traffic flow to eliminate stop-and-go driving through optimal coordination through traffic bottlenecks \cite{stern2018dissipation}.

In the late $1960s$, Levine and Athans \cite{Levine1966,Athans1969} proposed an optimal control framework for coordinating two groups of vehicles at merging roadways. Since then, substantial research efforts have been reported in the literature proposing optimal coordination of CAVs in different traffic scenarios such as merging roadways, roundabouts, speed reduction zones, and urban intersections. Among different traffic scenarios, intersections are the most challenging from a safety perspective, as an average of one-quarter of traffic fatalities and roughly half of all traffic injuries are attributed to intersections \cite{FHWA1}. 

Some work has considered a reservation-based approach for coordination of CAVs at a signal-free intersection, which requires CAVs to reserve a space-time slot inside the intersection. 
Dresner and Stone \cite{Dresner2008} first introduced this scheme based on a first-in-first-out (FIFO) queuing policy. In a sequel paper, Hausknecht et al. \cite{hausknecht2011autonomous} extended this scheme to the network of interconnected intersections aimed at exploring the best route to navigate a CAV arriving at the intersection to minimize its delay through the network.  Jin et al.\cite{jin2012multi} further developed the idea of a reservation-based scheme for signal-free intersection and relaxed the FIFO queuing policy. By relaxing FIFO, they showed that their approach resulted in better performance compared to the previous reservation-based schemes based on FIFO. 

Several studies focusing more on safety proposed a centralized coordination framework for CAVs at signal-free intersections. Lee and Park \cite{Lee2012a} designed a controller to minimize the total length of overlapped trajectories of CAVs inside the intersection. Bichiou and Rakha \cite{bichiou2018developing} considered minimizing travel time jointly with control efforts for $M$ closest CAVs to the intersection. Although the authors showed improvement in fuel efficiency and travel time, their approach takes $2-5$ minutes to find the optimal control actions for $M=4$ rendering it inapplicable for real-time implementation. 
Xu et al. \cite{xu2019cooperative} presented a centralized controller to find vehicles' crossing order at a signal-free intersection based on the heuristic tree search methods. 
Du et al. \cite{du2018hierarchical} introduced a tri-level coordination framework for CAVs at multiple intersections. Employing a consensus algorithm, each intersection derives the desired speed limit in the top level to balance the traffic density over multiple intersections.  In the middle level, the centralized controllers generate each vehicle's reference velocity, minimizing the deviation from the desired speed limit subject to lateral safety at the intersections. Finally, in the last level, each vehicle employs fast model predictive control (MPC) to track the reference velocity while avoiding rear-end collision. 

A number of research efforts have recently developed an optimal decentralized control framework for coordinating CAVs at a signal-free intersection.
Malikopoulos et al. \cite{Malikopoulos2017} presented a bi-level decentralized coordination framework for CAVs at a signal-free intersection addressing the throughput maximization and energy minimization problems. Using FIFO queuing policy, in the throughput maximization problem, each CAV computes its arrival time at the area of potential lateral collisions called merging zone. In the energy minimization problem, each CAV obtains its optimal acceleration/deceleration inside the control zone subject to speed and control constraints. The authors considered no turning maneuvers at the intersection and restricted the CAVs to travel with constant speed inside the merging zone.
Neglecting left/right turns, Mahbub et al. \cite{Mahbub2019ACC} provided the analytical unconstrained solution for two adjacent signal-free intersections.
In a follow-up paper to \cite{Malikopoulos2017}, Malikopoulos and Zhao \cite{malikopoulos2019ACC} further enhanced the framework by presenting an analytical solution for the speed-dependent rear-end safety constraint. 
Relying on FIFO queuing policy, Zhang and Cassandras \cite{zhang2019decentralized} presented a single-level decentralized coordination framework by formulating the objective function of each CAV to jointly minimize travel time and control effort with considering the minimum distance rear-end safety constraint. The authors provided the analytical solution for speed-dependent rear-end safety constraint in \cite{zhang2019joint}. 

Other research efforts have used scheduling theory to address the signal-free intersection problem   \cite{Colombo2015,Ahn2014,DeCampos2015a,Ahn2016,fayazi2018mixed,guney2020scheduling,yu2019corridor}.
Colombo and Del Vecchio \cite{Colombo2015} designed an intersection controller for a human driver which only intervenes and overrides the driver's control action when necessary, i.e., acting as a supervisory controller. They demonstrated that determining whether a state belongs to the maximal safe, controlled invariant set is equivalent to solving a scheduling problem. Ahn et al. \cite{Ahn2014} extended these results to include uncontrolled human drivers. In a sequel paper, Ahn and Del Vecchio \cite{Ahn2016} solved the supervisory problem for the first-order dynamics without considering the rear-end collision avoidance constraint using a mixed-integer linear program (MILP). Considering first-order dynamics and assuming an imposed speed inside the merging zone, Yu et al. \cite{yu2019corridor} formulated the coordination problem of CAVs at multiple intersections as a centralized MILP, the solution of which yields the trajectory of each CAV, along with the lane-changing maneuver decision, minimizing total travel time.
Fayazi and Vahidi \cite{fayazi2018mixed} considered a centralized intersection controller that constantly solves a scheduling problem using MILP for the arriving vehicles and passes the optimal arrival time for CAVs, thereby reducing stopping at the intersection and improving safety. 

There are several other efforts which have used MPC \cite{mirheli2019consensus,hult2018optimal,kamal2014vehicle}, fuzzy logic \cite{Onieva2012}, navigation function \cite{Makarem2012} to investigate coordination of CAVs at signal-free intersections. A thorough discussion of research efforts in the area of control and coordination of CAVs can be found in \cite{guanetti2018control} and \cite{Rios-Torres2017}.

\subsection{Contributions of This Paper}
A closer look at the literature on coordination of CAVs at signal-free intersections reveals that only a limited number of papers address coordination of CAVs at adjacent intersections. In such interconnected intersections, applying approaches of a single isolated intersection may result in sub-optimal, or even infeasible, solutions for CAVs. This is because the downstream intersection effect on the upstream intersection is not considered, and thus the roads connecting the two intersections can become easily congested. In addition, it is common for isolated intersections to consider a FIFO queuing policy to find the sequence of CAVs to enter the merging zone \cite{Malikopoulos2017,zhang2019decentralized,Rios-Torres2017}. However, considering two intersections together with the same paradigm as of a single intersection results in unnecessary slowdowns of the CAVs. Therefore, for two intersections that are closely distanced, not only we should not consider each intersection in isolation, but we also need a new paradigm for coordinating CAVs in these traffic scenarios.

In earlier work \cite{chalaki2020TCST}, we established a bi-level energy-optimal coordination framework for CAVs at multiple adjacent intersections without left/right turns focusing on minimizing energy consumption of CAVs. In the upper level, for each CAV, we presented a recursive algorithm to find the energy-optimal arrival time at each intersection along with the optimal lane to occupy. Given the solution of the upper-level optimization problem, we formulated an optimal control problem with interior-point constraints, the solution of which yields the energy optimal control input.

In this paper, we present a hierarchical decentralized coordination framework for CAVs at two adjacent intersections consisting of two levels. In the upper level, we formulate a decentralized scheduling problem for each CAV, which can be solved by using MILP upon entering the control zone. The solution of the upper-level problem yields the minimum travel time while satisfying safety constraints. The solution of the upper-level problem becomes the inputs of the low-level problem. In the low level, we formulate an optimal control problem for each CAV, the solution of which yields the energy optimal control input.
The contributions of this paper are: (1) the development of a hierarchical optimization framework to coordinate CAVs at two adjacent intersections considering all traffic movements aimed at decreasing both delay and travel time of each CAV; (2) a decentralized scheduling scheme for the upper-level problem considering state and control constraints that relaxes the strict FIFO queuing policy; (3) a complete, closed-form solution of the low-level optimization problem including the speed-dependent rear-end safety constraint and state and control constraints; and (4) a demonstration of the effectiveness of our approach through extensive numerical simulations including all possible paths for CAVs in two adjacent intersections and comparison with signalized intersections, centralized scheduling, and FIFO queuing policy. A limited-scope analysis of the hierarchical framework for the unconstrained solution with set constant speed inside the merging zones was presented in \cite{chalaki2019}.

From the technical perspective, although this paper and our earlier work \cite{chalaki2020TCST} address a similar problem, they are different from each other on major aspects such as problem formulation, solution approach, and results as follows: \\
    (1) In this paper, we consider two adjacent intersections including every possible path. In \cite{chalaki2020TCST}, we considered multiple multi-lane adjacent intersections; however, we limited our analysis to the cases that no left/right turns are allowed.\\
    (2) In this paper, we partition the area around two adjacent intersections into different zones, and assume that the speed of each CAV at the boundary of zones within the merging zones is given and is equal to $v_{\text{merge}}$.  The time that a CAV is inside each merging zone depends on the solution of the upper-level scheduling problem.
    However, in \cite{chalaki2020TCST}, we imposed a constant average speed inside the merging zone resulting in traveling at the merging zone with constant time.\\
    (3) The upper-level coordination framework in this paper is profoundly different from the one proposed in \cite{chalaki2020TCST}. The former is concerned with finding the arrival time at each partitioned zone on the CAV's path aimed at minimizing total travel time through formulating a scheduling problem, while the latter is formulated to find the energy-optimal arrival time at each merging zone and optimal lane to occupy through a recursive algorithm.\\
    (4) Similar to \cite{chalaki2020TCST}, in this paper, we follow Hamiltonian analysis to derive the closed-form analytical solution to the low-level problem. However, two major differences set the two low-level problems apart and make their Hamiltonian analysis fundamentally different. First, in this paper, the speed at the boundaries is defined, but in \cite{chalaki2020TCST}, speed at the initial time is only given. Thus, the problem defined at \cite{chalaki2020TCST} requires to satisfy new sets of optimality conditions due to the interior-point constraints which makes it different from the analysis here. Second, in contrast to \cite{chalaki2020TCST}, the rear-end safety constraint in this paper is speed dependent, which sets its Hamiltonian analysis completely apart from the minimum safe distance rear-end safety constraint discussed in \cite{chalaki2020TCST}.\\
    (5)  Moreover, in order to evaluate the performance of our approaches, we conducted different numerical simulations in both papers. In this paper, we provide simulation results for two scenarios. In the first scenario, we investigate coordination of CAVs at two adjacent intersections considering different traffic volumes, and then compare the results with the baseline scenario consisting of two-phase traffic signals. In the second scenario, we further evaluate the performance of our upper-level scheduling approach compared to the centralized scheduling and FIFO queuing policy. On the other hand, in \cite{chalaki2020TCST}, we studied the implications of our proposed coordination framework for CAVs under different traffic volumes, and then compared the results with a baseline scenario consisting of two-phase traffic signals. Moreover, we provided simulation results for symmetric and asymmetric adjacent intersections.

\subsection{Comparison with Related work}
To the best of our knowledge, this is the first attempt to establish a coordination framework for adjacent intersections aimed at minimizing total travel time. Therefore, we believe that this paper advances the state of the art in the following ways. First, in contrast to other efforts that investigated two intersections in isolation \cite{Colombo2014}, our framework presents a scheduling-based approach to consider the effects of intersections' interdependence and include effects of the downstream intersection on the upstream intersection. Second, our framework is not limited to straight paths \cite{Mahbub2019ACC,chalaki2020TCST} and does not exclude merging or splitting paths \cite{Colombo2014}. 
  Third, in several research efforts, the lateral safety was ensured through a strict FIFO queuing policy \cite{Malikopoulos2017,zhang2019decentralized,zhang2019joint,bichiou2018developing} or a centralized controller \cite{xu2019cooperative,kamal2014vehicle,du2018hierarchical,guney2020scheduling,Gregoire2014a,hult2018optimal}. In contrast, the decentralized upper-level scheduling problem in our framework relaxes the strict FIFO queuing policy. Namely, we demonstrate how our proposed framework outperforms the FIFO queuing policy through numerical simulations.
Finally, our framework in this paper, relaxes the assumptions of constant speed \cite{Malikopoulos2017,chalaki2019} and constant travel time \cite{chalaki2020TCST} inside the merging zone.
  
\subsection{Organization of This Paper}
The organization of this paper is as follows. 
In Section \ref{sec:PF}, we provide a detailed exposition of the modeling framework and the formulation of both low-level and upper-level optimization problems, while in Section \ref{sec:Solu}, we derive the corresponding solutions. We demonstrate the effectiveness of our approach through simulation in Section \ref{sec:SimRes}. Finally, in Section \ref{sec:conc}, we draw concluding remarks and discuss potential directions for future research.

\section{Problem Formulation} \label{sec:PF}
We consider two adjacent intersections shown in Fig. \ref{fig:2} which are closely distanced from each other. A "coordinator" stores information about the intersections' geometric parameters, the paths of the CAVs crossing the intersections, and the planned trajectories of CAVs.

The coordinator does not make any decision and it only acts as a database among the CAVs. The coordinator can be a physical infrastructure such as a drone, road site unit, or a cloud storage. In the remainder of this paper, we use the drone as one realization of the coordinator. We define the areas at which lateral collision inside the control zone may occur as merging zones.

\subsection{Modeling Framework}
Let $N(t)\in\mathbb{N}$ be the total number of CAVs entered the control zone by time $t\in\mathbb{R}^{+}$ and $\mathcal{N}(t)=\{1,\ldots,N(t)\}$ be the queue that designates the order that each CAV entered the control zone. Upon entering the control zone, each CAV  is assigned an integer $N(t)+1$ by the drone. If two, or more, CAVs enter the control zone at the same time, the CAV with a shorter path receives lower index in the queue; however, if the length of their path is the same, then their index is chosen arbitrarily. Finally, each CAV removes itself from $\mathcal{N}(t)$ when it exits the control zone. When there is no CAV inside the control zone, the queue $\mathcal{N}(t)$ is reset to zero.

We partition the roads around the intersections into $n_z\in\mathbb{N}$ zones where each zone has a unique integer index that belongs to the set $\mathcal{M}=\{1,\dots,n_z\}$. Although the number $n_z$ of partitions is arbitrary, choosing a big number increases the burden of computation for the scheduling problem since each CAV $i\in\mathcal{N}(t)$ needs to find its arrival time at each zone. We consider each road connecting to the merging zone to be a single zone. Similarly, we partition each merging zone into four smaller zones (Fig. \ref{fig:2}). Without being restrictive in our analysis, the total number of zones in the two intersections considered here (Fig. \ref{fig:2}) is $n_z=22$. We should note that zones are numbered arbitrarily. 
\begin{definition}
When CAV $i \in\mathcal{N}(t)$ enters the control zone, it creates a tuple of the zones $\mathcal{I}_i:=[m_1,\ldots,m_n]$, $m_n\in\mathcal{M}$, $n\in\mathbb{N}$, defined as the ``path'' of CAV $i$, where $m_1$ and $m_n$ denote the first and last zone on its path respectively, that $i$ needs to cross until it exits the control zone. 
\end{definition}
\begin{definition}
For each CAV $i\in\mathcal{N}(t)$ upon entering the control zone, we define the set $\mathcal{C}_{i,j}$ of conflict zones with CAV $j\in\mathcal{N}(t)$, which is present in the control zone ($j<i$),
\begin{gather}\label{1a}
\mathcal{C}_{i,j}=\{m ~|~m\in\:\mathcal{M},~ m\in\:\mathcal{I}_i \: , \:m\in\:\mathcal{I}_j \}.
\end{gather}
\end{definition}
\begin{figure}
\centering
\includegraphics[width=0.99\linewidth]{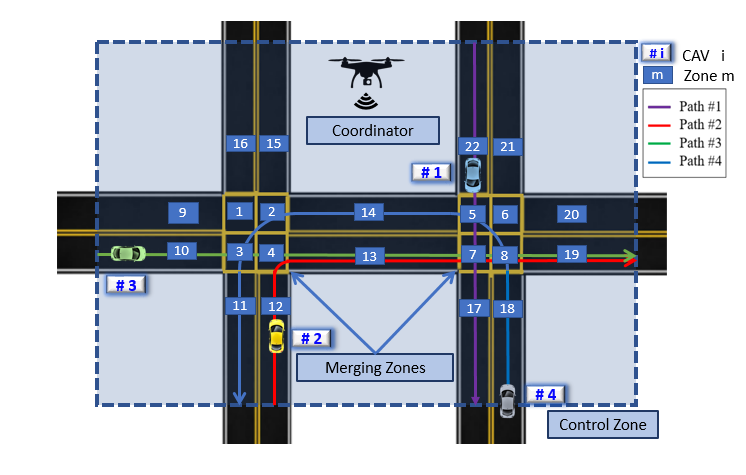}
\caption{
Two interconnected intersections with a drone as a coordinator. Zones numbered topologically and the fixed path for each CAV is shown.}
\label{fig:2}%
\end{figure}
For example in Fig. \ref{fig:2}, \mbox{CAV~\#$3$} has the following conflict tuples with \mbox{CAV~\#$1$} and \mbox{\#$2$} respectively: $\mathcal{C}_{3,1}=\{7\}$ and  $\mathcal{C}_{3,2}=\{4,13,7,8,19\}$.
\subsection{Vehicle model and assumptions}
We model the dynamics of each CAV $i\in\mathcal{N}(t)$ as a double integrator
\begin{gather}\label{27a}
\dot{p}_i(t)=v_i(t),\\
\dot{v}_i(t)=u_i(t),\nonumber
\end{gather}
where $p_{i}(t)\in\mathcal{P}_{i}$, $v_{i}(t)\in\mathcal{V}_{i}$, and
$u_{i}(t)\in\mathcal{U}_{i}$ denote position, speed and acceleration at $t\in\mathbb{R}^{+}$. Let $\mathbf{x}_{i}(t)=\left[p_{i}(t) , v_{i}(t)\right]^\top$ and $u_{i}(t)$ be the state
and control input of CAV $i$ at time $t$ respectively.
Let $t_{i}^{0}$ and $t_{i}^{f}$ be the time that CAV $i\in\mathcal{N}(t)$
enters and exits the control zone respectively, and $\mathbf{x}_{i}^{0}=\left[p_{i}(t_{i}^{0}), v_{i}(t_{i}^{0}) \right] ^\top$ be its initial state.
For each CAV $i\in\mathcal{N}(t)$, the control input and speed are bounded with the following constraints
\begin{equation}\label{uconstraint}
    u_{i,\min}\leq u_i(t)\leq u_{i,\max},
\end{equation}
\begin{equation}\label{vconstraint}
    0\leq v_{\min}\leq v_i(t)\leq v_{\max},
\end{equation}
where $u_{i,\min},u_{i,\max}$ are the minimum and maximum control inputs for each CAV $i\in\mathcal{N}(t)$, and $v_{\min},v_{\max}$ are the minimum and maximum speed limits respectively. Without loss of generality, we do not consider diversity among CAVs' maximum  and minimum control input. Thus, to this end, we set $u_{i,\min}=u_{\min}$ and $u_{i,\max}=u_{\max}$.
The sets $\mathcal{P}_{i}$,
$\mathcal{V}_{i}$ and $\mathcal{U}_{i}$, $i\in\mathcal{N}(t),$
are complete and totally bounded subsets of $\mathbb{R}$.

\begin{definition}\
Let CAV $k\in\mathcal{N}(t)$ be the preceding vehicle of CAV $i\in\mathcal{N}(t)$ in zone $m\in\:\mathcal{M}$. The distance, $d^m(p_k(t),p_i(t))$, between $i$ and $k$ in zone $m$ is defined as
\begin{equation}
d^m(p_k(t),p_i(t))=(p_k(t)-p_k(T_k^{^{m}}))-(p_i(t)-p_i(T_i^{^{m}})),
\end{equation}
where $p_k(T_k^{^{m}}),p_i(T_i^{^{m}})\in \mathbb{R}^+$ correspond to the distances from the entry point of the control zone to the entry point of the conflict zone $m$ for CAV $k$ and $i$ respectively. If no such CAV $k$ leads CAV $i$ at zone $m$, then we let $d^m(\cdot)\to\infty$. Note that, $p_k(T_k^{^{m}})$ and $p_i(T_i^{^{m}})$ depends on the geometry of the control zone and intersections. 
\end{definition}

To ensure the absence of rear-end collision between CAV $i\in\mathcal{N}(t)$ and the preceding CAV $k\in\mathcal{N}(t)$ in zone $m\in\:\mathcal{I}_i$, while  $m\in\mathcal{I}_k $, we impose the following rear-end safety constraint
\begin{equation}\label{RearEndCons}
  d^m(p_k(t),p_i(t))\geq \delta_i(t),~t\in[T_i^{^{m}},T^{^{m^\prime}}_i],
\end{equation}
where $T_i^{^{m}}$ and $T^{^{m^\prime}}_i$ are the entry time at and exit time from from zone $m$ of CAV $i$ respectively, and $\delta_i(t)$ is a predefined safe distance. 
The minimum safe distance $\delta_i(t) $ is a function of speed
\begin{equation}
    \delta_i(t) = \gamma + \varphi v_i(t),~t\in[T_i^{^{m}},T^{^{m^\prime}}_i],
\end{equation}
where $\gamma$ is the standstill distance, and $\varphi$ is the reaction time.

In our modeling framework described above, we impose the following assumptions:
\begin{assumption}\label{speedatBndMergingZone}
The speed of each CAV $i\in \mathcal{N}(t)$ at the boundary of zone $m\in\mathcal{M}$ in the merging zones is given and is equal to $v_{\text{merge}}$.
\end{assumption}
\begin{assumption}\label{feassibleAssum}
None of the state, control and safety constraints is active for each CAV $i\in \mathcal{N}(t)$ at the entry of the control zone.
\end{assumption}

The first assumption can be relaxed by estimating the speed at the boundaries of each zone in the upper-level problem. The second assumption is imposed to ensure that the initial state and control input are feasible. This is a reasonable assumption since CAVs are automated, and so there is no compelling reason for them to activate any of the constraints by the time they enter the control zone.

\subsection{Upper-level Problem: Scheduling}  
The objective of each CAV inside the control zone is to derive the optimal control input (acceleration/deceleration) aimed at minimizing travel time and improving traffic throughput. In the upper-level scheduling problem, each CAV $i\in\mathcal{N}(t)$ computes its arrival time to each zone $m\in\mathcal{I}_i$ that minimizes its total travel time inside the control zone and guarantees lateral safety constraints. 

Scheduling is a decision-making process that addresses the optimal allocation of resources to tasks over given time periods \cite{pinedo2016scheduling}.
 Thus, in what follows, we use scheduling theory to find the time that CAV $i\in\mathcal{N}(t)$ has to reach the zone $m\in\mathcal{I}_i$.
Each zone $m\in\mathcal{M}$ represents a "resource," and CAVs crossing this zone are the "jobs" assigned to the resource. 
\begin{definition}\label{DEF: scheduling tuple}
The time that a CAV $i\in\mathcal{N}(t)$  enters a zone $m\in\mathcal{I}_i$ is called "schedule" and is denoted by $T_i^{^{m}}\in\mathbb{R}^+$.
For CAV $i\in\mathcal{N}(t),$ we define a "schedule tuple,"  
\begin{equation}\label{2a1}
\mathcal{T}_{i}=[T_i^{^{m}} ~|~ m\in\mathcal{I}_i].
\end{equation}
\end{definition}
For example, the schedule tuple of CAV \#$1$ in Fig. \ref{fig:2} is  $\mathcal{T}_{1}=[{T}_{1}^{^{22}},{T}_{1}^{^{5}},{T}_{1}^{^{7}},{T}_{1}^{^{17}}]$.

For each zone $m\in\mathcal{I}_i$, $i\in\mathcal{N}(t),$ the schedule $T_i^{^{m}}\in\mathbb{R}^+$ is bounded by
\begin{equation}\label{Schedulecons1}
     T_i^{^{\underaccent{\bar}{m}}}+R_i^{^{\underaccent{\bar}{m}}}\leq T_i^{^{m}}\leq T_i^{^{\underaccent{\bar}{m}}}+D_i^{^{\underaccent{\bar}{m}}},\\
\end{equation}
where $\underaccent{\bar}{m}\in\mathcal{I}_i$ is the zone right before zone $m\in\mathcal{I}_i$, $T_i^{^{\underaccent{\bar}{m}}}$ is the time that CAV $i$ enters the zone $\underaccent{\bar}{m}$, and $R_i^{^{\underaccent{\bar}{m}}}\in\mathbb{R}^+$ and $D_i^{^{\underaccent{\bar}{m}}}\in\mathbb{R}^+$ are the shortest and latest feasible times that it takes for CAV $i\in\mathcal{N}(t)$ to travel through the zone $\underaccent{\bar}{m}$ respectively.  $R_i^{^{\underaccent{\bar}{m}}}\in\mathbb{R}^+$ and $D_i^{^{\underaccent{\bar}{m}}}\in\mathbb{R}^+$ are called the \textit{release time} and the \textit{deadline} of the job respectively. 
\begin{remark}

The exit time, $T^{^{m^\prime}}_i,$ of CAV $i\in\mathcal{N}(t)$ from zone $m\in\mathcal{I}_i$ is equal to the entry time to zone ${\underaccent{}{\bar{m}}}\in\mathcal{I}_i$, which is the zone that CAV $i$ crosses right after zone $m$.
\begin{equation}\label{exit=enter+1}
    T^{^{m^\prime}}_i=T_i^{^{\underaccent{}{\bar{m}}}}.
\end{equation}
\end{remark}
\begin{definition}\label{gamma}
 For each CAV $i\in\mathcal{N}(t)$, we define the set $\Gamma_i$ of all feasible time headways which do not violate the rear-end safety constraint \eqref{RearEndCons} at the entry of all zones $m \in\mathcal{I}_i$.
\end{definition}

\begin{definition}\label{safetyConst}
 For each CAV $i\in\mathcal{N}(t)$ and $j\in\mathcal{N}(t)$, $j<i,$ the safety constraint at the entry of zone $m\in\mathcal{C}_{i,j}$ can be restated as 
\begin{equation}\label{Schedulecons2}
  |T_i^{^{m}}-T_j^{^{m}}|\geq h,
\end{equation}
where $h\in\Gamma_i$ is the minimum time headway to avoid lateral collision.
\end{definition}

\begin{remark}
Definition \ref{safetyConst} relaxes the FIFO queuing policy for entering  zone $m\in\mathcal{M}$ by restricting the absolute value of the difference between the two schedules, rather than just enforcing $T_i^{^{m}} -T_j^{^{m}}\geq h$.
\end{remark}

\begin{problem}\label{Problem1} \textit{(Scheduling problem)}
For each CAV $i\in\mathcal{N}(t)$ with schedule tuple $\mathcal{T}_{i}$ and minimum time headway $h\in\Gamma_i$, the scheduling problem is formulated as follows
\begin{equation}\label{Scheduling}
\begin{array}{ll}
\min\limits_{\mathcal{T}_{i}} \quad J^{[1]}_i(\mathcal{T}_{i})=t_i^f(\mathcal{T}_{i}),\\
 \text{subject to:}~ (\ref{Schedulecons1}), (\ref{Schedulecons2}).\\
\end{array}  
\end{equation}
\end{problem}
\begin{remark}
In Problem \ref{Problem1}, the time $t_i^f$ that each CAV $i$ exits the control zone is a function of the schedule tuple $\mathcal{T}_{i}$ as implied by  \eqref{Schedulecons1}, which relates the arrival time at each zone to the arrival time at its previous zone.
\end{remark}
Upon entering the control zone, CAV $i$ solves the scheduling problem that yields its time-optimal arrival time at each zone. Then, it shares the schedule tuples with the drone.
Consider, for example (see Fig. \ref{fig:2}), \mbox{CAV \#$3$} with $\mathcal{I}_3=[10,3,4,13,7,8,19]$, $\mathcal{C}_{3,1}=\{7\}$ and  $\mathcal{C}_{3,2}=\{4,13,7,8,19\}$. The constraint (\ref{Schedulecons1}) for each zone $m\in\mathcal{I}_3$ is
\begin{align}
         &t_3^{0}+R_3^{^{10}}\label{eq:firstzonebound}\leq T_3^{^3}\leq t_3^{0}+D_3^{^{10}},\\
         &T_3^{^3}+R_3^{^3}\leq T_3^{^4}\leq T_3^{^3}+D_3^{^3},\\
         &T_3^{^4}+R_3^{^4}\leq T_3^{^{13}}\leq T_3^{^4}+D_3^{^4},\\
         &T_3^{^{13}}+R_3^{^{13}}\leq T_3^{^7}\leq T_3^{^{13}}+D_3^{^{13}},\\
         &T_3^{^7}+R_3^{^7}\leq T_3^{^8}\leq T_3^{^7}+D_3^{^7},\\
         &T_3^{^8}+R_3^{^8}\leq T_3^{^{19}}\leq T_3^{^8}+D_3^{^8},\\
         &T_3^{^{19}}+R_3^{^{19}}\leq t_3^{f}\leq T_3^{^{19}}+D_3^{^{19}}\label{eq:tfbounded}.
\end{align}
Note that the time \mbox{CAV \#$3$} enters the zone $\#10$, $T_3^{^{10}}$ is equal to the time that CAV $\#$ 3 enters the control zone $t_3^0$.
From the safety constraint (\ref{Schedulecons2}) for $m\in\mathcal{C}_{3,1}$ and $m\in\mathcal{C}_{3,2}$ we have 
\begin{align}
  &|T_3^{^7}-T_1^{^7}|\geq h,\\
  &|T_3^{^4}-T_2^{^4}|\geq h,\\
  &|T_3^{^{13}}-T_2^{^{13}}|\geq h,\\
  &|T_3^{^7}-T_2^{^7}|\geq h,\\
  &|T_3^{^8}-T_2^{^8}|\geq h,\\
  &|T_3^{^{19}}-T_2^{^{19}}|\geq h,
\end{align}
where the schedule tuples of CAV \#$1$ and \#$2$ are accessible through the drone. CAV $i\in\mathcal{N}(t)$ derives the release time and the deadline of each zone $m\in\mathcal{I}_i$ prior to solving the scheduling problem (Problem \ref{Problem1}). CAV \#$3$ above, for example, computes $R_3^{^m}$ and $D_3^{^m}$ for all $m\in\mathcal{I}_3$, and then it solves the scheduling problem, the solution of which yields the tuple $\mathcal{T}_{3}$. Next, we formulate the problems that yield the release time and deadline respectively.

\begin{problem}\label{problem2} \textit{(Release time problem)}
For each CAV $i\in\mathcal{N}(t)$ and each zone \mbox{$m\in\mathcal{I}_i$}, the release time $R_i^{^{m}}$ is derived by the following optimization problem
\begin{equation}\label{releasetimeProblem} 
\begin{array}{ll}
\min\limits_{\textit{u}_i\in\mathcal{U}_i} J^{[2]}_i(u_i(t))=t^{e,m}_i(u_i(t))-t^{s,m}_i, \\
\text{subject to:}~(\ref{27a}), (\ref{uconstraint}), (\ref{vconstraint}),\\
\text{given }~ p_i(t^{s,m}_i), v_i(t^{s,m}_i), p_i(t^{e,m}_i), v_i(t^{e,m}_i), \\
\end{array}  
\end{equation} \end{problem}
\noindent where $t^{s,m}_i$ and $t^{e,m}_i$ are the time that CAV $i\in\mathcal{N}(t)$ enters and exits the zone $m\in\mathcal{I}_i$ respectively.
The optimal solution ${u}_i^\ast(t)$ of Problem \ref{problem2} yields the release time, $R_i^{^{m}} = t^{e,m}_i({u}_i^\ast(t))-t^{s,m}_i$, which is the shortest feasible time that it takes for CAV $i$ to travel through zone $m$ without considering safety.

\begin{problem}\label{problem2.bDeadline}\textit{(Deadline problem)}
For each CAV $i\in\mathcal{N}(t)$ and each zone \mbox{$m\in\mathcal{I}_i$}, the deadline $D_i^{^m}$ is derived by the following optimization problem
\begin{equation}\label{deadlineProblem}
\begin{array}{ll}
\max\limits_{\textit{u}_i\in\mathcal{U}_i} ~ J^{[3]}_i(u_i(t))=t^{e,m}_i(u_i(t))-t^{s,m}_i, \\
\text{subject to:}~(\ref{27a}), (\ref{uconstraint}), (\ref{vconstraint}),\\
\text{given }~ p_i(t^{s,m}_i), v_i(t^{s,m}_i), p_i(t^{e,m}_i), v_i(t^{e,m}_i). \\
\end{array}  
\end{equation} \end{problem}
The optimal solution ${u}_i^\ast(t)$ of Problem \ref{problem2.bDeadline} yields the deadline, $D_i^{^m} = t^{e,m}_i({u}_i^\ast(t))-t^{s,m}_i$, which is the latest feasible time that it takes for CAV $i$ to travel through zone $m$ without considering safety.
\begin{remark}
Note that in Problems \ref{problem2} and \ref{problem2.bDeadline}, we do not consider safety constraints. The only objective of these two problems is to find the feasible bound for arrival time at each zone $m$ in \eqref{Schedulecons1}, to form the scheduling problem (Problem \ref{Problem1}).
\end{remark}

\subsection{Low-level problem: Energy Minimization}
After solving the upper-level scheduling problem, the low-level problem yields for each CAV the minimum control input at each zone (acceleration/deceleration) that satisfies the schedule resulted from the upper-level problem. 
\begin{problem}\label{problem3}
For each CAV $i\in\mathcal{N}(t)$ and each zone \mbox{$m\in\mathcal{I}_i$}, the energy minimization problem is
\begin{equation}\label{EnergyOptimalProblem}
 \begin{array}{ll}
 \min\limits_{\textit{u}_i\in\mathcal{U}_i} J^{[4]}_i(u_i(t),T_i^{^{m}},T_i^{^{\underaccent{}{\bar{m}}}})= {\dfrac{1}{2}} \bigints_{~T_i^{^{m}}}^{T_i^{^{\underaccent{}{\bar{m}}}}} u_{i}^2(t)~dt, \\
        
\text{subject to:}~(\ref{27a}), (\ref{uconstraint}), (\ref{vconstraint}),(\ref{RearEndCons}),\\
\text{given }~p_i(T_i^{^{m}}), v_i(T_i^{^{m}}), p_i(T_i^{^{\underaccent{}{\bar{m}}}}), v_i(T_i^{^{\underaccent{}{\bar{m}}}}), T_i^{^{m}}, T_i^{^{\underaccent{}{\bar{m}}}},\\
           \end{array} 
\end{equation}\end{problem}
\noindent where $T_i^{^{m}}$ and $T_i^{^{\underaccent{}{\bar{m}}}}$ are the entry and exit time of CAV $i\in\mathcal{N}(t)$ from zone $m\in\mathcal{I}_i$, determined by the upper-level scheduling problem.
We demonstrate the system architecture in Fig. \ref{fig: system structure}

\begin{figure}
\centering
\includegraphics[width=0.99\linewidth]{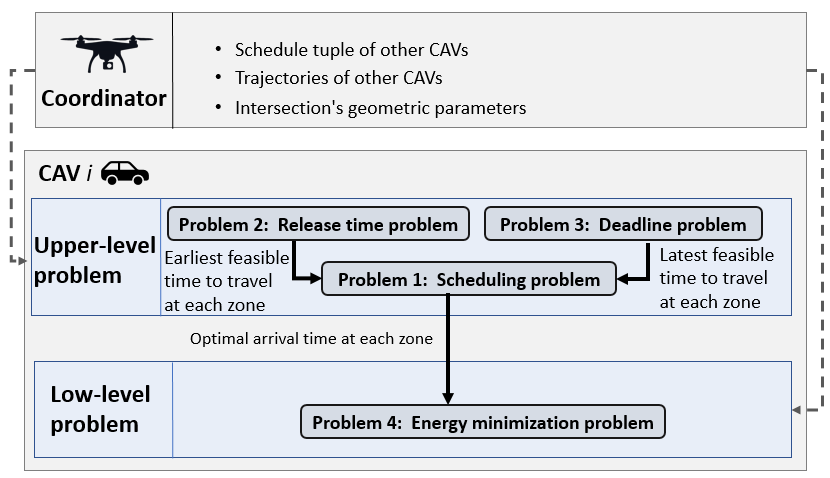}
\caption{
The hierarchical system architecture for CAV $i$ and drone (coordinator)}
\label{fig: system structure}%
\end{figure}

\section{Solution of low-level and upper-level problems} \label{sec:Solu}

In the previous section, we described the hierarchical optimization framework that consists of three upper-level problems and one low-level problem. 
Upon entering the control zone, each CAV is added to the queue $\mathcal{N}(t),$ and it solves the upper-level problems (Problems \ref{Problem1}, \ref{problem2} and \ref{problem2.bDeadline}) the solutions of which designate the optimal entry time to each zone along its path. In the upper-level problems, each CAV first derives the release time and deadline (Problems \ref{problem2} and \ref{problem2.bDeadline}) for each zone prior to solving the scheduling problem (Problem \ref{Problem1}). The outcome of the upper-level scheduling problem becomes the input of the low-level problem (Problem \ref{problem3}). In particular, in the low-level problem (Problem \ref{problem3}), each CAV derives the optimal control input (acceleration/deceleration) that minimizes energy consumption
at each zone of its path at the times specified in the upper-level (Problem \ref{Problem1}). 

To this end, to simplify notation, we use $p^{s}_i, v^{s}_i, p^{e}_i$ and $v^{e}_i$ instead of $p_i(t^{s,m}_i)$, $v_i(t^{s,m}_i)$, $p_i(t^{e,m}_i)$ and $v_i(t^{e,m}_i)$ respectively.

\subsection{Analytical solution of the release time and the \mbox{deadline}}
In this section, we provide the analytical closed-form solutions to Problems \ref{problem2} and \ref{problem2.bDeadline}, which each CAV $i\in\mathcal{N}(t)$ uses to formulate the scheduling problem (Problem \ref{Problem1}). One of the main advantages of deriving analytical solutions to Problem \ref{problem2} and \ref{problem2.bDeadline} is to improve the computational effort in the upper-level problem. 
For the analytical solution of the release time problem (Problem \ref{problem2}), we apply Hamiltonian analysis. 
For each CAV $i\in\mathcal{N}(t)$ the Hamiltonian function with the state and control constraints adjoined is  
\begin{equation}\label{hamil1}
\begin{aligned}
H_i&(t,p_i(t),v_i(t),u_i(t),\lambda_i(t))=1+\lambda_i^p v_i (t)+\lambda_i^v u_i(t)\\
+&\mu_i^a(u_i (t)-u_{{\max}})+\mu_i^b(u_{\min}-u_i(t))\\
+&\mu_i^c(v_i(t)-v_{{{\max}}})+\mu_i^d(v_{\min}-v_i(t)),\\
\end{aligned}
\end{equation}
where \(\lambda_i^p\) and \(\lambda_i^v\) are costates, and $\mu^\top$ is a vector of a lagrange multipliers: 
\begin{align}\label{20}
\mu_i^a&= \left\{ \begin{array}{ll}
         > 0\quad u_i(t)-u_{{\max}}=0\\
         =0 \quad u_i(t)-u_{{\max}}<0
\end{array} \right., \\
\mu_i^b&= \left\{ \begin{array}{ll}
         > 0\quad u_{\min}-u_i(t)=0\\
         =0 \quad u_{\min}-u_i(t)<0
\end{array} \right., \\
\mu_i^c&= \left\{ \begin{array}{ll}
         > 0\quad v_i(t)-v_{{\max}}=0\\
         =0 \quad v_i(t)-v_{{\max}}<0
\end{array} \right. ,\\
\mu_i^d&= \left\{ \begin{array}{ll}
         > 0\quad v_{\min}-v_i(t)=0\\
         =0 \quad v_{\min}-v_i(t)<0
\end{array} \right. .
\end{align}
The Euler-Lagrange equations become: 
\begin{align}
&\dot{\lambda_i^p}=-\frac{\partial H_i}{\partial p_i}=0,\label{24}\\
&\dot{\lambda_i^v}=-\frac{\partial H_i}{\partial v_i}=-\lambda_i^p-\mu_i^c+\mu_i^d .\label{25}
\end{align}

Similarly for the deadline problem (Problem \ref{problem2.bDeadline}), the Hamiltonian function is 
\begin{equation}\label{hamildeadline}
\begin{aligned}
H_i&(t,p_i(t),v_i(t),u_i(t),\lambda_i(t))=-1+\lambda_i^p v_i (t)+\lambda_i^v u_i(t)\\
+&\mu_i^a(u_i (t)-u_{{\max}})+\mu_i^b(u_{\min}-u_i(t))\\
+&\mu_i^c(v_i(t)-v_{{{\max}}})+\mu_i^d(v_{\min}-v_i(t)).
\end{aligned}
\end{equation}

\subsubsection{State constraints are not active}
First, we consider the case where the state constraint \eqref{vconstraint} does not become active, hence $\mu_i^c = \mu_i^d=0$.

\begin{lemma}\label{Lemma-oneSP-U}
The sign of the optimal control input of the release time problem (Problem \ref{problem2}) for zone $m$, when the state constraint is not active, can change at most once, and it is equal to either: (1) $u_i(t) =u_{\min}$, or (2) $u_i(t)=u_{\max}$ , or (3) $u_i(t)=u_{\max}$ and then it switches to $u_i(t)=u_{\min}$.
\end{lemma}
\begin{proof}\phantom{\qedhere}
See Appendix \ref{AA}. 
\end{proof}

\begin{lemma}\label{lemma-speed-position-intermediate}
In Problem 2, let $\mathbf{x}_{i}^{s}=[p^s_i,v^s_i]^\top$ and $\mathbf{x}_{i}^{e}=[p^e_i,v^e_i]^\top$ be the initial and final states of CAV $i\in\mathcal{N}(t)$ traveling in zone $m\in\mathcal{I}_i$. Let $\mathbf{x}_{i}^{c}=[p^c_i,v^c_i]^\top$ be the intermediate state at the time $t^{c,m}_i$ that the control input changes sign. Then,
\begin{equation}\label{zz}
p^{c}_i=\frac{{v^e_i}^2-{v^s_i}^2+2(u_{{\max}}p^s_i-u_{\min}p^e_i)}{2(u_{{\max}}-u_{\min})},
\end{equation}
\begin{equation}\label{z}
v^{c}_i=\sqrt{{v^s_i}^2+2u_{{\max}}\cdot(p_i^c-p_i^s)}.
\end{equation}
\end{lemma}
\vspace{1mm}
\begin{proof}
From (\ref{27a}) and Lemma \ref{Lemma-oneSP-U}, the intermediate states are found by solving the following system of equations 
\begin{equation}\label{z2}
 \left\{ \begin{array}{ll}
{v^c_i}^2-{v^s_i}^2=2u_{{\max}}\cdot(p^c_i-p^s_i)\\
{v^e_i}^2-{v^c_i}^2=2u_{\min}\cdot(p^e_i-p^c_i)
           \end{array}, \right. 
\end{equation}
which yields (\ref{zz}) and (\ref{z}).
\end{proof}

\begin{proposition}\label{prop-endtime-intermediateTime}
The release time of CAV $i\in\mathcal{N}(t)$ traveling in zone $m\in\mathcal{I}_i$, when the state constraint is not active, is
\begin{equation}\label{processtimeEq}
 R_{i}^{^m}=\frac{v^c_i -v^s_i}{u_{{\max}}}+\frac{v^e_i-v^c_i}{u_{\min}}.
\end{equation}
\end{proposition}
\begin{proof}
When $u_i(t)=u_{\max}$ for all $t\in[t^{s,m}_i,t^{c,m}_i]$ and $u_i(t)=u_{\min}$ for all $t\in[t^{c,m}_i,t^{e,m}_i]$, where $t^{c,m}_i$ is the time that the control input changes sign, the total time traveled inside the zone $m$ can be found by integrating~(\ref{27a}), hence 
\begin{equation}\label{lem3proof2}
\begin{array}{ll}
v^c_i -v^s_i=u_{{\max}}\cdot (t^{c,m}_i-t^{s,m}_i), & \forall~ t\in[t^{s,m}_i,t^{c,m}_i],\\
v^e_i -v^c_i=u_{\min}\cdot(t^{e,m}_i-t^{c,m}_i), & \forall~ t\in[t^{c,m}_i,t^{e,m}_i].
\end{array}
\end{equation}
Solving (\ref{lem3proof2}) for $t^{c,m}_i$ and $t^{e,m}_i$, we have
\begin{equation}\label{tc}
t^{c,m}_i=\frac{v^c_i -v^s_i}{u_{{\max}}}+t^{s,m}_i,
\end{equation}
\begin{equation}\label{te}
t^{e,m}_i=\frac{v^c_i -v^s_i}{u_{{\max}}}+\frac{v^e_i-v^c_i}{u_{\min}}+t^{s,m}_i.
\end{equation}
Substituting $t_i^{e,m}$ into $R_{i}^{^m}=t_i^{e,m} - t_i^{s,m}$ yields \eqref{processtimeEq}. 
\end{proof}

\begin{lemma}\label{Lemma-deadline-control-input}
The optimal control input of the deadline problem (Problem \ref{problem2.bDeadline}) for zone $m\in\mathcal{I}_i$, when the state constraint is not active, changes sign at most once, and it is equal to either: (1) $u_i(t) =u_{\min}$, or (2) $u_i(t)=u_{\max}$ , or (3) $u_i(t)=u_{\min}$ and then it switches to $u_i(t)=u_{\max}$.
\end{lemma}
\begin{proof}
The proof is similar to the proof of Lemma \ref{Lemma-oneSP-U}, and thus, it is omitted.
\end{proof}

\begin{proposition}\label{prop-Deadline}
Let $\mathbf{x}_{i}^{s}=[p^s_i,v^s_i]^\top$ and $\mathbf{x}_{i}^{e}=[p^e_i,v^e_i]^\top$ be the initial and final states of CAV $i\in\mathcal{N}(t)$ traveling in zone $m\in\mathcal{I}_i$. Let $\mathbf{x}_{i}^{c}=[p^c_i,v^c_i]^\top$ be the intermediate state at the time $t^{c,m}_i$ that the control input changes sign. Then, the deadline of CAV $i$ traveling in zone $m\in\mathcal{I}_i$ for the unconstrained case is
\begin{equation}\label{deadline3}
D_{i}^{^m}=\frac{v^c_i -v^s_i}{u_{{\min}}}+\frac{v^e_i-v^c_i}{u_{\max}},
\end{equation}
\begin{equation}\label{deadline2}
v^{c}_i=\sqrt{{v^s_i}^2+2u_{{\min}}\cdot(p_i^c-p_i^s)},
\end{equation}
\begin{equation}\label{deadline1}
p^{c}_i=\frac{{v^e_i}^2-{v^s_i}^2+2(u_{{\min}}p^s_i-u_{\max}p^e_i)}{2(u_{{\max}}-u_{\min})}.
\end{equation}
\end{proposition}
\begin{proof}
The control input of $i\in\mathcal{N}(t)$ in zone $m\in\mathcal{I}_i$ consists of two arcs, i.e., decelerating with $u_i(t)={u_{{\min}}}$ and accelerating with $u_i(t)={u_{\max}}$. Following similar arguments to Lemma \ref{lemma-speed-position-intermediate}, we derive \eqref{deadline1} and \eqref{deadline2}. The total time traveled inside the zone $m$ can be found by integrating~(\ref{27a}), hence 
\begin{equation}\label{prop-Deadlinesystems}
\begin{array}{ll}
v^c_i -v^s_i=u_{{\min}}\cdot (t^{c,m}_i-t^{s,m}_i), & \forall~ t\in[t^{s,m}_i,t^{c,m}_i],\\
v^e_i -v^c_i=u_{\max}\cdot(t^{e,m}_i-t^{c,m}_i), & \forall~ t\in[t^{c,m}_i,t^{e,m}_i].
\end{array}
\end{equation}
Solving (\ref{prop-Deadlinesystems}) for $t^{c,m}_i$ and $t^{e,m}_i$, we have 
\begin{equation}\label{tcDead}
t^{c,m}_i=\frac{v^c_i -v^s_i}{u_{{\min}}}+t^{s,m}_i,
\end{equation}
\begin{equation}\label{teDead}
t^{e,m}_i=\frac{v^c_i -v^s_i}{u_{{\min}}}+\frac{v^e_i-v^c_i}{u_{\max}}+t^{s,m}_i.
\end{equation}
Substituting $t_i^{e,m}$ into $D_{i}^{^m}=t_i^{e,m} - t_i^{s,m}$ \eqref{deadline3} follows. 
\end{proof}

\subsubsection{State constraints are active}
Next, we consider the cases where the speed constraints become active. 
\begin{theorem}\label{theorm-noactive-withoutSP}
In Problems \ref{problem2} and \ref{problem2.bDeadline}, if there is no change on the sign of the control input, then none of the speed constraints becomes active. 
\end{theorem}
\begin{proof}
We consider the two cases that there is no change on the sign of the control input, i.e., case 1: $u_i(t) = u_{\min}$, case 2: $u_i(t) = u_{{\max}}$. 

Case 1: For all $t<t^\prime \in[t^{s,m}_i,t^{e,m}_i]$, we have
\begin{equation}
v_i(t) > v_i(t^\prime).
\end{equation}
Hence, the minimum and maximum speed can only occur at $t^{e,m}_i$ and $t^{s,m}_i$ respectively, namely
\begin{align}
v_i^e \leq v_i(t),&~\forall~t \in[t^{s,m}_i,t^{e,m}_i],\\
v_i(t) \leq v_i^s,&~\forall~t \in[t^{s,m}_i,t^{e,m}_i].
\end{align}
However, from the Assumptions \ref{speedatBndMergingZone} and \ref{feassibleAssum}, we have
\begin{equation}
v_{\min}<v_i^e\leq v_i(t)\leq v_i^s<v_{\max} ,\quad\forall~ t \in[t^{s,m}_i,t^{e,m}_i].\\
\end{equation}

Case 2: Following similar arguments to Case 1, we have
\begin{equation}
v_{\min}<v_i^s\leq v_i(t)\leq v_i^e<v_{\max} ,\quad\forall~ t \in[t^{s,m}_i,t^{e,m}_i].\\
\end{equation}
\end{proof}

\begin{corollary}\label{onlyVmax}
For CAV $i\in\mathcal{N}(t)$ in zone $m\in\mathcal{I}_i$, the unconstrained solution of the release time problem (Problem \ref{problem2}) can not activate the constrained arc ~$v_i(t) = v_{\min}$.\end{corollary}

\begin{proof}
From Theorem \ref{theorm-noactive-withoutSP}, we know that if there is no change on the sign of the control input, then none of the speed constraints becomes active. Let's consider that the control input changes sign at $t^{c,m}_i \in [t^{s,m}_i,t^{e,m}_i]$, thus
\begin{align}
u_i(t) &= u_{\max}>0 \Rightarrow
v_i^s \leq v_i(t) ,~\forall~ t \in[t^{s,m}_i,t^{c,m}_i],\\
u_i(t) &= u_{\min}<0 \Rightarrow v_i^e \leq v_i(t) ,~\forall~ t \in[t^{c,m}_i,t^{e,m}_i].
\end{align}
It follows that the minimum speed of CAV $i$ for all $t\in[t^{s,m}_i,t^{e,m}_i]$ is either $v_i^e$ or $v_i^s$. 
From the Assumptions \ref{speedatBndMergingZone} and \ref{feassibleAssum}, state constraints are not active at the entry and exit of the zones, and the proof is complete. \end{proof}

One can verify whether the unconstrained solution of CAV $i$ leads to violation of the speed constraint $v_i(t)\leq v_{\max}$ in zone $m$, by checking the speed at the interior point $v^c_i$ found from \eqref{z}.
If the unconstrained solution violates the speed constraint $v_i(t)\leq v_{\max}$, then the solution exits the unconstrained arc at time $\tau_1$, and enters the constrained arc $v_i(t)=v_{\max}$. Then the unconstrained arc is pieced together with the constrained arc $v_i(t)= v_{{\max}}$, and we re-solve the problem with the two arcs pieced together. The two arcs yield a set of algebraic equations that are solved simultaneously using the boundary conditions and interior conditions between the arcs. Since the speed at the boundary of zones do not activate the speed constraint, the solution cannot stay at the constrained arc $v_i(t)= v_{\max}$ and it must exit the constrained arc $v_i(t)= v_{{\max}}$ at time $\tau_2$. The unconstrained and constrained arcs are pieced together, and we re-solve the problem consisting of the three arcs.

\begin{theorem}\label{theorem-endtime-intermediateTime-con}
The release time of CAV $i\in\mathcal{N}(t)$ traveling in zone $m\in\mathcal{I}_i$ when the constraint $v_i(t)=v_{\max}$ is active is
\begin{equation}\label{processtimeEqCons}
 R_{i}^{^m}= \frac{a_i+b_i}{2u_{\min}~u_{{\max}}~v_{{\max}}},
\end{equation}
where
\begin{align}
    a_i &= {v^s_i}^2~u_{{\min}}-{v^e_i}^2~u_{{\max}}+(u_{{\min}}-u_{{\max}})~v_{{\max}}^2,\\
    b_i &= 2u_{{\min}}~u_{{\max}}(p^e_i-p^s_i)+2v_{{\max}}~(v^e_i~u_{{\max}}-v^s_i~u_{{\min}}).
\end{align}

\end{theorem}
\begin{proof}
See Appendix \ref{B}.
\end{proof}

\begin{remark}
Similar to Corollary \ref{onlyVmax}, for CAV $i\in\mathcal{N}(t)$ in zone $m\in\mathcal{I}_i$, the unconstrained solution of the deadline problem (Problem \ref{problem2.bDeadline}) can not activate the constrained arc ~$v_i(t) = v_{\max}$.
\end{remark}

\begin{proposition}\label{prop-DeadlineConstrained}
Let $\tau_1$ and $\tau_2$ be the time that CAV $i\in\mathcal{N}(t)$ enters the constrained arc $v_i(t) = v_{\min}$, while it is in zone $m\in\mathcal{I}_i$. Let $[ p^{s}_i, v^{s}_i]^\top$ and $[p^{e}_i, v^{e}_i]^\top$ be the initial and final states of CAV $i$ in zone $m$ respectively. Then, the deadline of CAV $i$ to exit zone $m$ is 
\begin{equation}\label{deadlineConstrained}
    D_i^m = \frac{v^{e}_i- v_{{\min}}}{u_{{\max}}} + \tau_2 - t_i^{s,m},
\end{equation}
where 
\begin{align}
\tau_2 &= \frac{p_i(\tau_2)-p_i(\tau_1)}{v_{{\min}}} + \tau_1,\\
  p_i(\tau_2) &= \frac{v_{{\min}}^2-{v^{e}_i}^2}{2~u_{{\max}}}+p^{e}_i,\\
 p_i(\tau_1) &= \frac{v_{{\min}}^2-{v^{s}_i}^2}{2~u_{{\min}}}+p^{s}_i,\\
\tau_1 &= \frac{-v^{s}_i+ v_{{\min}}}{u_{{\min}}} + t_i^{s,m}.
\end{align}
\end{proposition}
\begin{proof}
The proof is similar to the proof of Theorem \ref{theorem-endtime-intermediateTime-con}, and thus, it is omitted.
\end{proof}

\subsection{Solution of the scheduling problem (Problem \ref{Problem1})}
As we described earlier, at the entry of the control zone, CAV $i\in\mathcal{N}(t)$ computes the release time and deadline for each zone $m\in\mathcal{I}_i$. Then, it solves the scheduling problem (Problem \ref{Problem1}), the solution of which determines the schedule tuple $\mathcal{T}_i$ (Definition \ref{DEF: scheduling tuple}) aimed at minimizing the time $t_i^f$ that $i$ exits the control zone.   
In the scheduling problem (Problem \ref{Problem1}) of CAV $i$, for each zone $m$ which belongs to the conflict set $\mathcal{C}_{i,j}$ where $j<i \in\mathcal{N}(t)$, we impose the safety constraint \eqref{Schedulecons2} (Definition \ref{safetyConst}) stated as  
\begin{equation}\label{disjunctive constraint}
  T_i^{^{m}}-T_j^{^m} \geq h,
\end{equation}
OR \begin{equation}\label{disjunctive constraint2}
  -(T_i^{^{m}}-T_j^{^m}) \geq h,
\end{equation}
which is a disjunctive constraint due to the OR statement, and also determines the order of entry at zone $m$. By introducing a binary variable $B^{^{m}}_{i,j}\in \{0,1\}$ and big number $M\in\mathbb{R}^{+}$ \cite{grossmann2012generalized}, we rewrite the disjunctive constraints \eqref{Schedulecons2} as two separate constraints as following
\begin{align}\label{disjunctive constraint3}
  (T_i^{^{m}}-T_j^{^m}) + B^{^{m}}_{i,j} \cdot M &\geq h,\\
  -(T_i^{^{m}}-T_j^{^m}) + (1-B^{^{m}}_{i,j}) \cdot M &\geq h.
\end{align}
However, there are two cases to consider in handling the safety constraint; case 1: CAV $i$ and CAV $j$ follow the same path, and case 2: CAV $i$ and CAV $j$ have different paths that merges together.
In case 1, we have $\mathcal{C}_{i,j}=\mathcal{I}_i=\mathcal{I}_j$ implying that they conflict on each zone of their path. Since CAV $i$ entered the control zone later, it cannot arrive at any zone earlier than CAV $j$, thus, we have $B_{i,j}^m=0$ for all $m\in\mathcal{C}_{i,j}$.
For the second case, let $\mathcal{C}_{i,j}=\{m_a,\dots,m_b\}$ be the conflict set between CAV $i$ and $j$, which have merging paths, $m_a$ be the first zone that there is a potential lateral collision, and $m_b$ be their last conflict zone with the potential rear-end collision. Since we relaxed the FIFO queuing policy, CAV $i$ can arrive at zone $m_1$ either before ($B_{i,j}^{m_1}=1$), or after CAV $j$ ($B_{i,j}^{m_1}=0$). However, to ensure the absence of the rear-end safety in the following zones after $m_a$, we need to have all the binary variables equal to $B_{i,j}^{m_a}$, i.e.,
\begin{align}\label{disjunctive constraint4}
B_{i,j}^{m_a} = \dots = B_{i,j}^{m_b},\quad B_{i,j}^{m_a}\in \{0,1 \}.
\end{align}

In addition, the arrival time at each zone $m$ is lower bounded with the arrival time and release time, and upper-bounded with arrival time and deadline of the previous zone $\underaccent{\bar}{m}$. Similarly, the exit time from the control zone $t_i^f$ is bounded by the arrival time, release time, and deadline of the last zone.

After transforming each safety constraint to two separate constraints augmented with a binary variable, we use a mixed-integer linear program (MILP) (IBM ILOG CPLEX \cite{IBMILOg}) to solve the scheduling problem. We  discuss the implications on computation effort of solving the MILP in Section \Rmnum{4}.

\subsection{Analytical solution of the energy minimization problem (Problem \ref{problem3}) }
One approach to address the inequality constraints, which are a function of state variables, is adjoining the $q$th-order state variable inequality constraint to the Hamiltonian function. The $q$th-order state variable inequality constraint can be found by taking the successive total time derivative of constraint and substitute \eqref{27a} for $\dot{\mathbf{x}}$, until we obtain an expression that is explicitly dependent on the control variable \cite{bryson1975applied}. 
For each CAV $i\in\mathcal{N}(t)$, with CAV $k\in\mathcal{N}(t)$ positioned immediately in front of it, the Hamiltonian is
\begin{gather}\label{c1}
H_i(t,p_i(t),v_i(t),u_i(t),\lambda(t))=\frac{1}{2}u_i(t)^2+\lambda_i^p v_i (t)+\lambda_i^v u_i(t)\nonumber\\
+\mu_i^a(u_i (t)-u_{{\max}})+\mu_i^b(u_{\min}-u_i(t))\nonumber\\
+\mu_i^c(u_i(t))+\mu_i^d(-u_i(t))\nonumber\\
+\mu_i^s(v_i(t)-v_k(t)+\varphi u_i(t)),
\end{gather}
where \(\lambda_i^p\) and \(\lambda_i^v\) are the costates, and $\mu^\top$ is a vector of Lagrange multipliers.
\\
\subsubsection{State and control constraints are not active}
If the state and control constraints are not active, \(\mu_i^a=\mu_i^b=\mu_i^c=\mu_i^d=\mu_i^s=0\) and from \cite{Malikopoulos2017} the solution is
\begin{equation}\label{27}
 u_i^{\ast}=a_it+b_i,
\end{equation}
by substituting \eqref{27} in \eqref{27a} we have
\begin{align}\label{27b}
&v_i^{\ast}=\frac{1}{2}a_it^2+b_it+c_i,\\
&p_i^{\ast}=\frac{1}{6}a_it^3+\frac{1}{2}b_it^2+c_it+d_i.
\end{align}
In the above equations \(a_i,b_i,c_i,d_i\) are constants of integration, which are found by substituting the initial and final states $p_i(T_i^{^{m}}), v_i(T_i^{^{m}})$, $p_i(T_i^{^{\underaccent{}{\bar{m}}}})$ and $v_i(T_i^{^{\underaccent{}{\bar{m}}}})$ in zone $m\in\mathcal{I}_i$ .
Thus, a system of equations in the form of \(\textbf{T}_i \textbf{b}_i=\textbf{q}_i\), is
\begin{equation}\label{energy-optimal-matrix}
  \begin{bmatrix}
   \frac{1}{6}\left(T_i^{^{m}}\right)^3&
   \frac{1}{2}(T_i^{^{m}})^2&
   T_i^{^{m}}&
   1\\
   \\
   \dfrac{1}{2}\left(T_i^{^{m}}\right)^2&
   T_i^{^{m}}&
   1&
   0\\
   \\
   \dfrac{1}{6}\left(T_i^{^{\underaccent{}{\bar{m}}}}\right)^3&
   \dfrac{1}{2}\left(T_i^{^{\underaccent{}{\bar{m}}}}\right)^2&
   T_i^{^{\underaccent{}{\bar{m}}}}&
   1\\
   \\
   \dfrac{1}{2}\left(T_i^{^{\underaccent{}{\bar{m}}}}\right)^2&
   T_i^{^{\underaccent{}{\bar{m}}}}&
   1&
   0\\
   \end{bmatrix}
   \cdot
     \begin{bmatrix}
   a_i\\
   b_i\\
   c_i\\
   d_i
   \end{bmatrix}
   =
     \begin{bmatrix}
   p_i(T_i^{^{m}})\\
   v_i(T_i^{^{m}})\\
   p_i(T_i^{^{\bar{m}}})\\
   v_i(T_i^{^{\underaccent{}{\bar{m}}}})
   \end{bmatrix}.
\end{equation}
Note that since (\ref{energy-optimal-matrix}) can be computed
online, the controller may re-evaluate the four constants at any time $t\in[T_i^{^{m}},T_i^{^{\underaccent{}{\bar{m}}}}]$ and update (\ref{27}).

There are different cases that can happen and activate either the state or control constraints. Next, we consider the cases that CAV only travels on the constrained arcs except the rear-end safety \eqref{RearEndCons}.
\subsubsection{CAV travels on different constrained arcs}

\textbf{Case 1:}  CAV $i\in\mathcal{N}(t)$ enters the constrained arc $u_i(t)=u_{{\max}}$ at $T_i^{^{m}}$. Then, it moves to the constrained arc $u_i(t)=u_{{\min}}$ at $\tau_1$ and stays on it until $T_i^{^{\underaccent{}{\bar{m}}}}$.
\begin{theorem}\label{unconst-control-constraint becomes active}
Let $T_i^{^{m}}$ and $T_i^{^{\underaccent{}{\bar{m}}}}$ be the schedules of CAV $i\in\mathcal{N}(t)$ for zones $m , \underaccent{}{\bar{m}} \in\mathcal{I}_i$ respectively, where zone ${\underaccent{}{\bar{m}}}$ is right after $m$. In zone $m$, the CAV $i$ first enters the constrained arc $u_i(t)=u_{{\max}}$ and then the arc $u_i(t)=u_{{\min}}$, if the speed constraint does not become active in zone $m$ and 
\begin{equation}\label{schedule=release}
   T_i^{^{\underaccent{}{\bar{m}}}}-T_i^{^{m}}=R_i^{^m}.
\end{equation}
\end{theorem}
\begin{proof}
Let CAV $i\in \mathcal{N}(t)$ enter and and exit the zone $m\in\mathcal{I}_i$ at $T_i^{^{m}}$  and $T^{^{m^\prime}}_i$ respectively. Substituting \eqref{exit=enter+1} into $T_i^{^{\underaccent{}{\bar{m}}}}-T_i^{^{m}}=R_i^{^m}$, we have $T^{^{m^\prime}}_i-T_i^{^{m}}=R_i^{^m}$.
Thus, CAV $i$ is traveling at the earliest feasible time in zone $m$. Since the speed constraint is not active, the solution is equivalent to the solution of the release time problem (Problem \ref{problem2}) when the speed constraint is not active. 
\end{proof}

In the above case, the optimal control input is 
\begin{equation}\label{Umax-Umin}
 u^*_i(t)=\left\{ \begin{array}{ll}
u_{\max}&, \mbox{if}\quad T_i^{^{m}}\leq t<\tau_1\\
u_{\min} &, \mbox{if}\quad \tau_1\leq t\leq T_i^{^{\underaccent{}{\bar{m}}}}\\
           \end{array}. \right. 
\end{equation}
Substituting (\ref{Umax-Umin}) in (\ref{27a}), we have 
\begin{align}
p_i^{\ast}(t)=&\frac{1}{2}u_{{\max}}t^2+b_it+c_i, \nonumber \\
v_i^{\ast}(t)=&u_{{\max}}t+b_i , & \forall~ t\in [T_i^{^{m}},\tau_1^-],\\
p_i^{\ast}(t)=&\frac{1}{2}u_{{\min}}t^2+d_it+e_i, \nonumber \\
v_i^{\ast}(t)=&u_{{\min}}t+d_i ,& \forall~ t\in [\tau_1^+,T_i^{^{\underaccent{}{\bar{m}}}}],
\end{align}
where $b_i,c_i,d_i$ and $e_i$ are constants of integration, which are found by using initial conditions $p_i(T_i^{^{m}}), v_i(T_i^{^{m}})$ and final conditions $p_i(T_i^{^{\underaccent{}{\bar{m}}}}), v_i(T_i^{^{\underaccent{}{\bar{m}}}})$ of the CAV in zone $m\in\mathcal{I}_i$. The switching point $\tau_1$ can be found from \eqref{tc}.

\textbf{Case 2:}  CAV $i$ enters the constrained arc  $u_i(t)=u_{{\max}}$ at $T_i^{^{m}}$, then it moves to the constrained arc $v_i(t)=v_{{\max}}$ at $t=\tau_1$. It exits the constrained arc $v_i(t)=v_{{\max}}$ at $t=\tau_2$, and it enters the constrained arc $u_i(t)=u_{{\min}}$ and stays on it until $T_i^{^{\underaccent{}{\bar{m}}}}$.

\begin{corollary}\label{corol.const-control-constraint becomes active}
Case $2$ is realized for CAV $i\in \mathcal{N}(t)$ travelling in zone $m\in\mathcal{I}_i$, if the speed constraint becomes active and \eqref{schedule=release} holds. 
\end{corollary}
\begin{proof}
The proof is similar to the proof of Theorem \ref{unconst-control-constraint becomes active}, and thus, it is omitted.
\end{proof}

In this case, the optimal control input is \begin{gather}\label{Umax-Vmax-Umin}
 u^*_i(t)=\left\{ \begin{array}{ll}
u_{\max}&, \mbox{if}\quad T_i^{^{m}}\leq t<\tau_1\\
0 &, \mbox{if}\quad \tau_1\leq t\leq \tau_2\\
u_{\min} &, \mbox{if}\quad \tau_2\leq t\leq T_i^{^{\underaccent{}{\bar{m}}}}\\
           \end{array}. \right. 
\end{gather}
Substituting (\ref{Umax-Vmax-Umin}) in (\ref{27a}), we have
\begin{align}
p_i^{\ast}(t)=&\frac{1}{2}u_{{\max}}t^2+b_it+c_i, \nonumber \\
v_i^{\ast}(t)=&u_{{\max}}t+b_i , &\forall~ t\in [T_i^{^{m}},\tau_1^-],\\
p_i^{\ast}(t)=&v_{{\max}}t+d_i, \nonumber \\
v_i^{\ast}(t)=&v_{{\max}}, &\forall~ t\in [\tau_1^+,\tau_2^-],\\
p_i^{\ast}(t)=&\frac{1}{2}u_{{\min}}t^2+e_it+f_i, \nonumber \\
v_i^{\ast}(t)=&u_{{\min}}t+e_i , &\forall~ t\in [\tau_2^+,T_i^{^{\underaccent{}{\bar{m}}}}],
\end{align}
where $b_i,c_i,d_i,e_i$ and $f_i$ are constants of integration, and time $\tau_1$ and $\tau_2$ are times that we move from one arc to another arc, which are found by using initial and final conditions in zone $m$ and continuity of states at $\tau_1$ and $\tau_2$.

\textbf{Case 3:} CAV $i\in\mathcal{N}(t)$ enters the constrained arc $u_i(t)=u_{{\min}}$ at $T_i^{^{m}}$. Then, it moves to the constrained arc $u_i(t)=u_{{\max}}$ at $\tau_1$ and stays on it until $T_i^{^{\underaccent{}{\bar{m}}}}$.
\begin{theorem}\label{deadline-Unconst-constraint becomes active}
Let $T_i^{^{m}}$ and $T_i^{^{\underaccent{}{\bar{m}}}}$ be the schedules of CAV $i\in\mathcal{N}(t)$ for zones $m , \underaccent{}{\bar{m}} \in\mathcal{I}_i$ respectively, where zone ${\underaccent{}{\bar{m}}}$ is right after zone $m$. In zone $m$, CAV $i$ first enters the constrained arc $u_i(t)=u_{{\min}}$, and then the constrained arc $u_i(t)=u_{{\max}}$, if the speed constraint does not become active in zone $m$ and 
\begin{equation}\label{schedule=deadline}
   T_i^{^{\underaccent{}{\bar{m}}}}-T_i^{^{m}}=D_i^m.
\end{equation}
\end{theorem}
\begin{proof}
The proof is similar to the proof of Theorem \ref{unconst-control-constraint becomes active}, and thus, it is omitted.
\end{proof}

In this case, the optimal control input is
\begin{gather}\label{Umax-umin}
 u^*_i(t)=\left\{ \begin{array}{ll}
u_{\min}&, \mbox{if}\quad T_i^{^{m}}\leq t<\tau_1\\
u_{\max} &, \mbox{if}\quad \tau_1\leq t\leq T_i^{^{\underaccent{}{\bar{m}}}}\\
           \end{array}. \right. 
\end{gather}
Substituting (\ref{Umax-umin}) in (\ref{27a}), we have
\begin{align}
p_i^{\ast}(t)=&\frac{1}{2}u_{{\min}}t^2+b_it+c_i, \nonumber \\
v_i^{\ast}(t)=&u_{{\max}}t+b_i ,&\forall~ t\in [T_i^{^{m}},\tau_1^-],\\
p_i^{\ast}(t)=&\frac{1}{2}u_{{\max}}t^2+d_it+e_i, \nonumber \\
v_i^{\ast}(t)=&u_{{\min}}t+d_i ,&\forall~ t\in [\tau_1^+,T_i^{^{\underaccent{}{\bar{m}}}}],
\end{align}
where $b_i,c_i,d_i$, and $e_i$ are integration constants, which can be computed by using initial and final conditions of CAV in zone $m\in\mathcal{I}_i$ respectively. The switching point $\tau_1$ can be found from \eqref{deadline3}.

\textbf{Case 4:} CAV $i$ enters the constrained arc  $u_i(t)=u_{{\min}}$ at $T_i^{^{m}}$, then it moves to the constrained arc $v_i(t)=v_{{\min}}$ at $t=\tau_1$. It exits the constrained arc $v_i(t)=v_{{\min}}$ at $t=\tau_2$, and it enters the constrained arc $u_i(t)=u_{{\max}}$ and stays on it until $T_i^{^{\underaccent{}{\bar{m}}}}$.

\begin{corollary}\label{Corol.deadline-Unconst-constraint becomes activ}
Case $4$ is realized for CAV $i\in \mathcal{N}(t)$ travelling in zone $m\in\mathcal{I}_i$, if the speed constraint becomes active in zone $m$ and \eqref{schedule=deadline} holds. 
\end{corollary}
\begin{proof}
The proof is similar to the proof of Theorem \ref{unconst-control-constraint becomes active}, and thus, it is omitted.
\end{proof}

In this case, the optimal control input is 
\begin{gather}\label{umin-Vmin-Umax}
 u^*_i(t)=\left\{ \begin{array}{ll}
u_{\min}&, \mbox{if}\quad T_i^{^{m}}\leq t<\tau_1\\
0 &, \mbox{if}\quad \tau_1\leq t\leq \tau_2\\
u_{\max} &, \mbox{if}\quad \tau_2\leq t\leq T_i^{^{\underaccent{}{\bar{m}}}}\\
           \end{array}. \right. 
\end{gather}
Substituting (\ref{umin-Vmin-Umax}) in (\ref{27a}), we have 
\begin{align}
p_i^{\ast}(t)=&\frac{1}{2}u_{{\min}}t^2+b_it+c_i, \nonumber \\
v_i^{\ast}(t)=&u_{{\min}}t+b_i , &\forall~ t\in [T_i^{^{m}},\tau_1^-],\\
p_i^{\ast}(t)=&v_{{\min}}t+d_i, \nonumber \\
v_i^{\ast}(t)=&v_{{\min}}, &\forall~ t\in [\tau_1^+,\tau_2^-],\\
p_i^{\ast}(t)=&\frac{1}{2}u_{{\max}}t^2+e_it+f_i, \nonumber \\
v_i^{\ast}(t)=&u_{{\max}}t+e_i , &\forall~ t\in [\tau_2^+,T_i^{^{\underaccent{}{\bar{m}}}}].
\end{align}

\subsubsection{CAV travels on a combination of constrained and unconstrained arcs}
Using \eqref{27}, we first start with the unconstrained solution of Problem \ref{problem3}. If the solution violates any of the speed \eqref{vconstraint} or control \eqref{uconstraint} constraints, then the unconstrained arc is pieced together with the arc corresponding to the violated constraint at unknown time $\tau_1$, and we re-solve the problem with the two arcs pieced together. The two arcs yield a set of algebraic equations which are solved simultaneously using the boundary conditions and interior conditions at $\tau_1$. If the resulting solution violates another constraint, then the last two arcs are pieced together with the arc corresponding to the new violated constraint, and we re-solve the problem with the three arcs pieced together at unknown times $\tau_1$ and $\tau_2$. The three arcs will yield a new set of algebraic equations that need to be solved simultaneously using the boundary conditions and interior conditions at $\tau_1$ and $\tau_2$. The process is repeated until the solution does not violate any other constraints, \cite{Malikopoulos2017}.

\subsubsection{CAV enters the safety constrained arc}
Let CAV $i$ enter and exit zone $m\in\mathcal{I}_i$ at $ T_i^{^{m}}$ and $ T_i^{^{\underaccent{}{\bar{m}}}}$ respectively. CAV $k$ is immediately positioned in front of CAV $i$ in zone $m$, and CAV $i$ activates the rear-end safety constraint \eqref{RearEndCons} at time $\tau_1\in[T_i^{^{m}},T_i^{^{\underaccent{}{\bar{m}}}}] $. We have two cases to consider: Case 1: CAV $i$ remains in the constrained arc until $T_i^{^{\underaccent{}{\bar{m}}}}$ or Case 2: CAV $i$ exits the constrained arc at $\tau_2\in[\tau_1,T_i^{^{\underaccent{}{\bar{m}}}}] $.

\begin{lemma}\label{lemma-notactivatingRearEnd}
Let CAV $i$ and $k\in\mathcal{N}(t)$ enter zone $m$ at time $T_i^{^{m}}$ and $T_k^{^{m}}$ respectively. Let CAV $k$ be immediately ahead of $i$. Then the rear-end safety constraint for CAV $i$ does not become active at the entry of zone $m$, if the minimum time headway $h\in \Gamma_i$, where
\begin{equation}\label{lemma-RearEndSet}
    \Gamma_i = \{t ~|~ \frac{1}{2}u_{\min}t^2+v_k(T_k^{^{m}})t-\varphi v_i(T_i^{^{m}})-\gamma_i >0 ,~\forall~ t \in\mathbb{R}^{+}\}.
\end{equation}
\end{lemma}
\begin{proof}\phantom{\qedhere}
See Appendix \ref{D}.
\end{proof}
\begin{lemma}\label{lemma-casesforRearEnd}

If the rear-end safety constraint becomes active for CAV $i\in\mathcal{N}(t)$ at $\tau_1 \in(T_i^{^{m}},T_i^{^{\underaccent{}{\bar{m}}}})$, then it must exit the rear-end safety constrained arc at $\tau_2\in[\tau_1,T_i^{^{\underaccent{}{\bar{m}}}})$.
\end{lemma}
\begin{proof}
From Lemma \ref{lemma-notactivatingRearEnd}, the rear-end safety constraint of CAV $i$ with a schedule $T_i^{^{m}}\in\mathcal{T}_i$ does not become active at the entry of zone $m\in\mathcal{I}_i$. Thus, if the rear-end safety constraint of CAV $i$ becomes active at $\tau_1$, it must exit the constrained arc before it exits zone $m$.
\end{proof}

Suppose CAV $i\in\mathcal{N}(t)$ enters the zone $m$ at time $t=T_i^{^m}$ and at some time $t=\tau_1$, the rear-end safety constraint with the vehicle $k$ becomes active until $t =\tau_2$, $d^m(p_k(t),p_i(t)) = \delta_i(t)$ for all $t\in[\tau_1,\tau_2]$, in this case $\mu_i^s\neq0$.
Let $N_i(t,\mathbf{x}_i(t))=(p^{\ast}_i(t)-p_i^{\ast}(T_i^{^{m}}))-(p^{\ast}_k(t)-p_k^{\ast}(T_k^{^{m}}))+\gamma + v^{\ast}_i(t)$. Note that, $p_i^{\ast}(T_i^{^{m}})$ and 
$p_k^{\ast}(T_k^{^{m}})$ are time-invariant and we can simplify the notation by defining $\Bar{\gamma} = -p_i^{\ast}(T_i^{^{m}})+p_k^{\ast}(T_k^{^{m}})+\gamma$. Thus, we have
\begin{equation}\label{tangency}
 N_i(t,\mathbf{x}_i(t))=p^{\ast}_i(t)-p^{\ast}_k(t)+\Bar{\gamma} + \varphi v^{\ast}_i(t).   
\end{equation}
 Since $N_i(t,\mathbf{x}_i(t)) = 0$ for $t\in[\tau_1,\tau_2]$, its first derivative, which is dependent on the optimal control input, should vanish in $t\in[\tau_1,\tau_2]$
\begin{equation}\label{1-storderSafety}
    N^{(1)}_i(t,\mathbf{x}_i(t))=v^{\ast}_i(t)-v^{\ast}_k(t) + \varphi u^{\ast}_i(t) = 0.
\end{equation}
By taking a time derivative from \eqref{1-storderSafety}, the optimal control input of CAV $i$, when rear-end safety constraint is active can be found from solving the following ODE.
\begin{equation}\label{27safety}
\dot{u}^{\ast}_i(t)+\frac{1}{\varphi}u^{\ast}_i(t)-\frac{1}{\varphi}u^{\ast}_k(t) = 0,~\forall~ t\in[\tau_1,\tau_2].
\end{equation}
The optimal solution need to satisfy the following jump conditions on costates upon entry the constrained arc at $t=\tau_1$
\begin{align}
&\lambda_i^p(\tau_1^-) = \lambda_i^p(\tau_1^+)+\pi_i\frac{\partial N_i}{\partial p_i} = \lambda_i^p(\tau_1^+)+\pi_i,\label{36safety}\\
&\lambda_i^v(\tau_1^-) = \lambda_i^v(\tau_1^+)+\pi_i\frac{\partial N_i}{\partial v_i} = \lambda_i^v(\tau_1^+)+\varphi\pi_i,\label{36safety}\\
&H_i(\tau_1^-) = H_i(\tau_1^+)-\pi_i\frac{\partial N_i}{\partial t} = H_i(\tau_1^+)+\pi_iv^{\ast}_k(t),\label{37safety}
\end{align}
where $\pi_i$ is a constant Lagrange multipliers, determined so that $N_i(\tau_1,\mathbf{x}_i(\tau_1))=0$.
At the exit point of the constrained arc we have
\begin{align}
\lambda_i^p(\tau_2^-) &= \lambda_i^p(\tau_2^+),\label{39safety}\\
\lambda_i^v(\tau_2^-) &= \lambda_i^v(\tau_2^+),\label{40safety}\\
H_i(\tau_2^-) &= H_i(\tau_2^+).\label{41safety}
\end{align}
As described earlier, the three arcs need to be solved simultaneously using initial and final conditions (speed and position), and interior conditions at unknown time $\tau_1$ and $\tau_2$ (continuity of speed and position, jump conditions \eqref{tangency}-\eqref{41safety}).
The complete analytical solution when the rear-end safety constraint becomes active has been presented in \cite{malikopoulos2019ACC, Malikopoulos2020}.

\section{Simulation results} \label{sec:SimRes}

To evaluate the effectiveness of the proposed framework to improve travel time and the traffic throughput, we investigate coordination of CAVs at two adjacent intersections considering different traffic volumes, and then compare the results with the baseline scenario consisting of two-phase traffic signals. We consider two adjacent intersections which are $100$ m apart. In addition, the length of each road connecting to the intersections is $300$ m, and the length of the merging zones are $30$ m. We used the following parameters for the simulation: $h = 1.5$ s, $v_{\min}=5$ m/s, $v_{\max}=25$ m/s, $v_{\text{merge}}=15$ m/s  $u_{\min}=-1$  m/s$^2$, $u_{\max}=1$ m/s$^2$, $\gamma = 5$ m, and $\varphi = 0.2$ s. 

For the first scenario, we consider CAVs enter the control zone with initial speed uniformly distributed between $13$ m/s to $16$ m/s from four conflicting paths shown in figure \ref{fig:2} with equal traffic volumes. 
We construct the baseline scenario with two-phase fixed-time traffic signals in PTV-VISSIM by considering all vehicles as human-driven and without any vehicle-to-vehicle communication. We use VISSIM built-in traffic signal optimizer to obtain the traffic signal timing. To emulate the driving behavior of real human-driven vehicles, we use a built-in car-following model (Wiedemann \cite{Wiedemann1974}) in PTV-VISSIM with default parameters. We let the speed limit for the baseline scenario to be $v_{\max}$. In the optimal-scenario, we use MATLAB to simulate our framework. To compare the optimal scenario with the baseline scenario, CAVs enter the control zone at the same time, speed, and path that they entered in the baseline scenario. Videos of the experiment can be found at the supplemental site, \url{https://sites.google.com/view/ud-ids-lab/TITS}.

Table \ref{tbl:average travel time} presents the average travel time of all CAVs inside the control zone for the baseline and optimal scenarios at different traffic volumes ranging from $400$ veh/h to $1200$ veh/h per path. For each traffic volume, we performed five simulations with different random seeds and averaged the results. Within our proposed framework average travel time has been decreased by $21\% -33\%$ compared to the baseline scenario. Relative frequency histogram of travel time of each CAV for traffic volume $1200$ veh/h for a randomly selected seed for the baseline and optimal scenarios are shown in Fig. \ref{fig:histTravelTime}. The optimal scenario has a high relative frequency of lower travel time compared to the baseline scenario. Travel time for $65\%$ of CAVs lies in the range [$40$s-$50$s] for the optimal scenario, whereas, travel time of vehicles for the baseline scenario has a higher variation, and only $20\%$ of vehicles' travel time is in the range [$20$s-$50$s]. In the optimal scenario, maximum travel time is in the range [$60$s-$70$s] compared to the maximum range [$120$s-$130$s] for the baseline scenario.      

\begin{table}[htbp]
\caption{Average travel time of vehicles in the optimal and baseline scenarios for different traffic volumes.}
\vspace{0.5em}
\centering
\begin{tabular}{c|c|c|c|c} \label{tbl:average travel time}
    Traffic volume & Average number& \multicolumn{2}{c}{Average travel time (s)}& Decrease
    \\
    (veh/h)   & of vehicles &  Baseline              & Optimal          & $\%$
    \\
    \toprule
        $400$  &14&  $51.89$  &   $40.81$  & $21$\\
        $600$  &18&  $57.52$  &   $41.86$  & $27$\\
        $800$  &25&  $63.30$  &   $43.26$  & $32$\\
        $1000$ &31&  $68.82$   &   $46.59$  & $32$\\
        $1200$ &36&  $72.15$  &   $48.53$  & $33$\\
\end{tabular}
\end{table}

\begin{figure}
\centering
\includegraphics[width=0.9\linewidth]{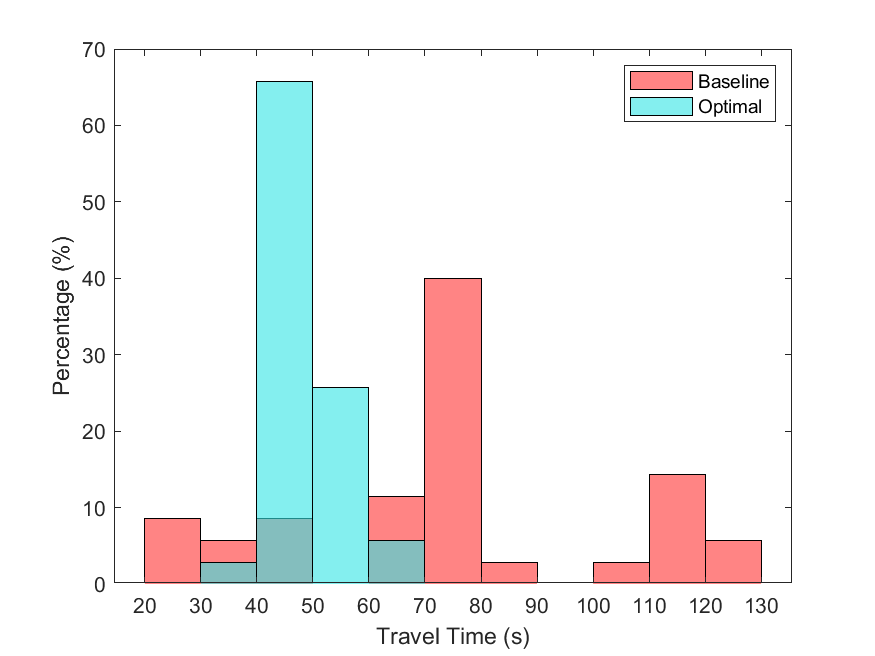}
\caption{A relative frequency histogram for travel time of each vehicle for the baseline and optimal scenarios with traffic volume 1200 veh/h. }
\label{fig:histTravelTime}%
\end{figure}

The instantaneous average, maximum and minimum speed of CAVs inside the control zone for the baseline and optimal scenarios with traffic volume 600 veh/h, 1000 veh/h, and 1200 veh/h for a randomly selected seed are illustrated in Fig \ref{fig:rangeSpeed}. The instantaneous minimum speed for all traffic volumes in the optimal scenario is positive and higher than $v_{\min}$ indicating smooth traffic flow, compared to the baseline scenario, which experiences much stopping due to the traffic lights.  Relative frequency histogram of the average speed of each CAV for traffic volume $1200$ veh/h for the baseline and optimal scenarios are plotted in Fig. \ref{fig:histMeanSpeed} on top of each other. Figure \ref{fig:histMeanSpeed} shows that, in the baseline scenario, the average speed of almost all CAVs are lower than the optimal scenario, and that the CAVs in the optimal scenario speed up and $70\%$ of CAVs achieve mean speed higher than $14$ m/s, whereas in the baseline scenario vehicles decrease their speed while $75\%$ of the vehicles have mean speed less than $12$ m/s.

To evaluate the fuel efficiency improvement through our framework, we use a polynomial meta-model proposed in \cite{kamal2012model}, which approximate the fuel consumption in ml/s as a function of speed and control input of a CAV and coefficients obtained from an engine torque-speed-efficiency map of a typical car. Table \ref{tbl:average fuel consumption} summarizes the average fuel rate and fuel consumption for the optimal and baseline scenarios at different traffic volumes for which five simulations with different random seeds were performed, and the results were averaged. Although our time-optimal framework results in higher average fuel rate compared to the baseline scenario, it leads to smaller fuel consumption as the traffic flow gets higher. Higher fuel rate in our approach is result of speeding to reach the speed imposed at the boundaries of zones in the merging zones, i.e., $v_{\text{merge}}$. Similarly, setting $v_{\text{merge}}$ at boundaries of zones may lead into a discontinuity of the control input at these points which can cause higher fuel rate, and passenger's discomfort. As it can be seen, there is a trade-off between minimizing travel time and energy-consumption. One may consider minimizing them jointly by formulating multi-objective cost function as proposed by \cite{zhang2019decentralized,zhang2019joint} for a single intersection, or using a recursive structure for the upper-level problem in \cite{Malikopoulos2017}
to trade off between the energy minimization and throughput maximization indirectly. 
\begin{table}[htbp]
\caption{Average fuel rate, and fuel consumption of vehicles in the optimal and baseline scenarios for different traffic volumes.}
\vspace{0.5em}
\centering
\begin{tabular}{c|c|c|c|c} \label{tbl:average fuel consumption}
    Traffic volume & \multicolumn{2}{c}{Average fuel rate (ml/s)}& \multicolumn{2}{c}{Average fuel consumption (l)}\\
    (veh/h)   &Baseline & Optimal&Baseline &  Optimal 
    \\
    \toprule
        $400$  &$1.53$&  $2.03$  &   $0.076$  & $0.083$\\
        $600$  &$1.45$&  $1.92$  &   $0.079$  & $0.080$\\
        $800$  &$1.35$&  $1.80$  &   $0.079$  & $0.077$\\
        $1000$ &$1.29$&  $1.62$   &  $0.079$  & $0.074$\\
        $1200$ &$1.24$&  $1.57$  &   $0.077$  & $0.075$\\
\end{tabular}
\end{table}

The mean and standard deviation of computation times of CAVs in solving the MILP in the scheduling problem (Problem \ref{Problem1}) for different traffic volumes is listed in Table \ref{tbl:computation}. For each traffic flow, we choose the seed with a maximum mean of computation time to report. It shows that the scheduling problem is computationally feasible and does not grow exponentially with increasing the traffic volume and number of CAVs. The mean computation time of CAVs in different traffic volume is in the range \mbox{[$21$ ms - $25.4$ ms]} with very small standard deviation. 

In the second scenario, we further evaluate the performance of our upper-level scheduling approach compared to the centralized scheduling and FIFO queuing policy. Namely, we consider all different possible paths in two adjacent intersections. We have six different origins or destination for vehicles, including northbound $1$ (NB$1$), northbound $2$ (NB$2$), eastbound (EB), westbound (WB),  southbound $1$ (SB$1$), and southbound $2$ (SB$2$). There are five possible paths for each origin, which in total would be $30$ possible paths. We assume that CAVs enter the control zone based on a Poisson process with a rate of $1$ second. We performed five simulations with random seeds for different numbers of vehicles and averaged the results.

To obtain maximum possible performance of the system in the centralized scheduling, we consider that the initial conditions of all CAVs, including their arrival time and their initial speed at the control zone, are known to the central controller. We then formulated a centralized scheduling problem aimed at minimizing average travel time of all CAVs subject to the constraints (\ref{Schedulecons1}) and (\ref{Schedulecons2}).
The solution to the centralized scheduling problem is schedule tuples of all CAVs, which then can be used as inputs for the low-level problem. On the other hand, in our upper-level scheduling formulation, each CAV solves the scheduling problem upon entering the control zone, and once the solution is derived, then the schedule tuple of the CAV does not change. When a new CAV $i$ enters the control zone, the schedule tuples of other CAVs are fixed since they entered earlier, and they already obtained their optimal schedule tuples. Using this sequential decision-making approach, we may sacrifice optimality in terms of minimizing total travel time since we assume that when a new vehicle arrives, it cannot change the schedule of the vehicles that arrived earlier. However, the optimality gap between the centralized and decentralized solution may be reduced by introducing a re-planning framework, to update the schedules of the vehicles as new vehicle arrives.

Figure \ref{fig:comparisonUpperLevel} illustrates the average travel time of all CAVs and computation time for our decentralized scheduling approach, centralized scheduling, and FIFO queuing policy. As it can be seen, there is a close gap between the average travel time achieved by our decentralized scheduling problem and the centralized scheduling problem. However, the computation complexity of the centralized scheduling problem renders it inapplicable for real-time implementation.
Moreover, there are significant delays in travel time caused by strict FIFO queuing policy leading into sub-optimal solutions for two adjacent intersections. Note that, by increasing the number of vehicles from $15$ to $75$, average travel time is changed from $41$ s to $44.6$ s, and from $40$ s to $42.9$ s for decentralized and centralized scheduling, respectively. However, this is not a significant change since the vehicles arrive at the control zone at the fixed flow rate based on the Poisson process with a rate of $1$ second.

\begin{figure}
\centering
\includegraphics[width=0.99\linewidth]{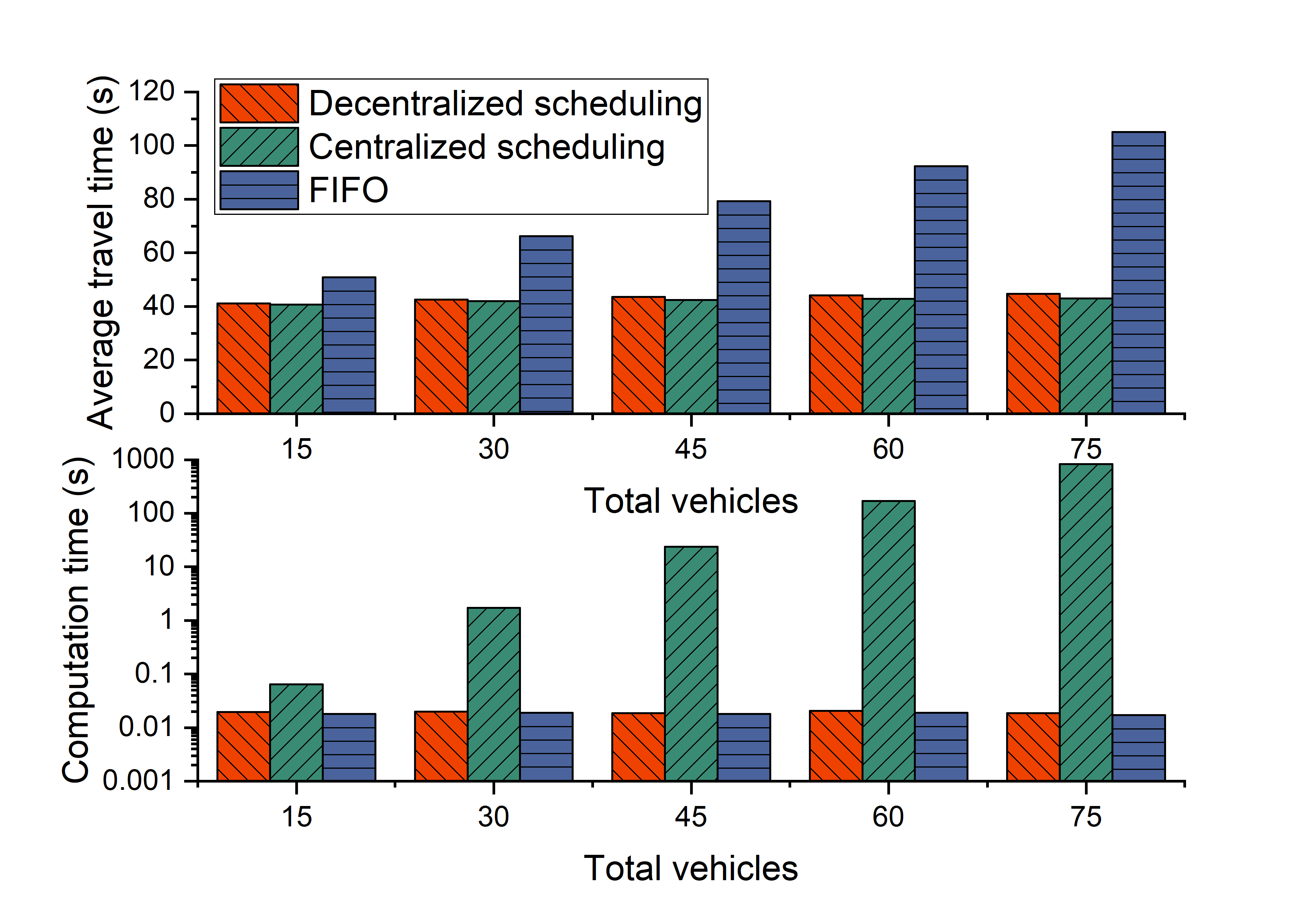}
\caption{Average travel time and computation time of decentralized scheduling, centralized scheduling, and FIFO queuing policy for different numbers of vehicles.}
\label{fig:comparisonUpperLevel}%
\end{figure}

To demonstrate the effects of flow rate on the performance of our approach, we evaluate coordination of $30$ vehicles arriving at all possible paths under different Poisson process arrival rates compared to the centralized scheduling and FIFO queuing policy. Similar to the previous scenarios, we performed $5$ simulations with random seeds for different arrival rates and averaged the results. Figure \ref{fig:comparisonUpperLevelDifferentArrival} shows the average travel time and average computation time for decentralized scheduling, centralized scheduling, and FIFO queuing policy. In Fig. \ref{fig:comparisonUpperLevelDifferentArrival}, smaller Poisson arrival rates correspond to the higher traffic volumes. In all three approaches, by decreasing traffic volume, average travel time increases; however, this change is more significant in the FIFO queuing policy. In the centralized scheduling approach, the computation complexity increases with the increase in traffic volume, making it difficult for real-time implementation.

\begin{figure}
\centering
\includegraphics[width=0.99\linewidth]{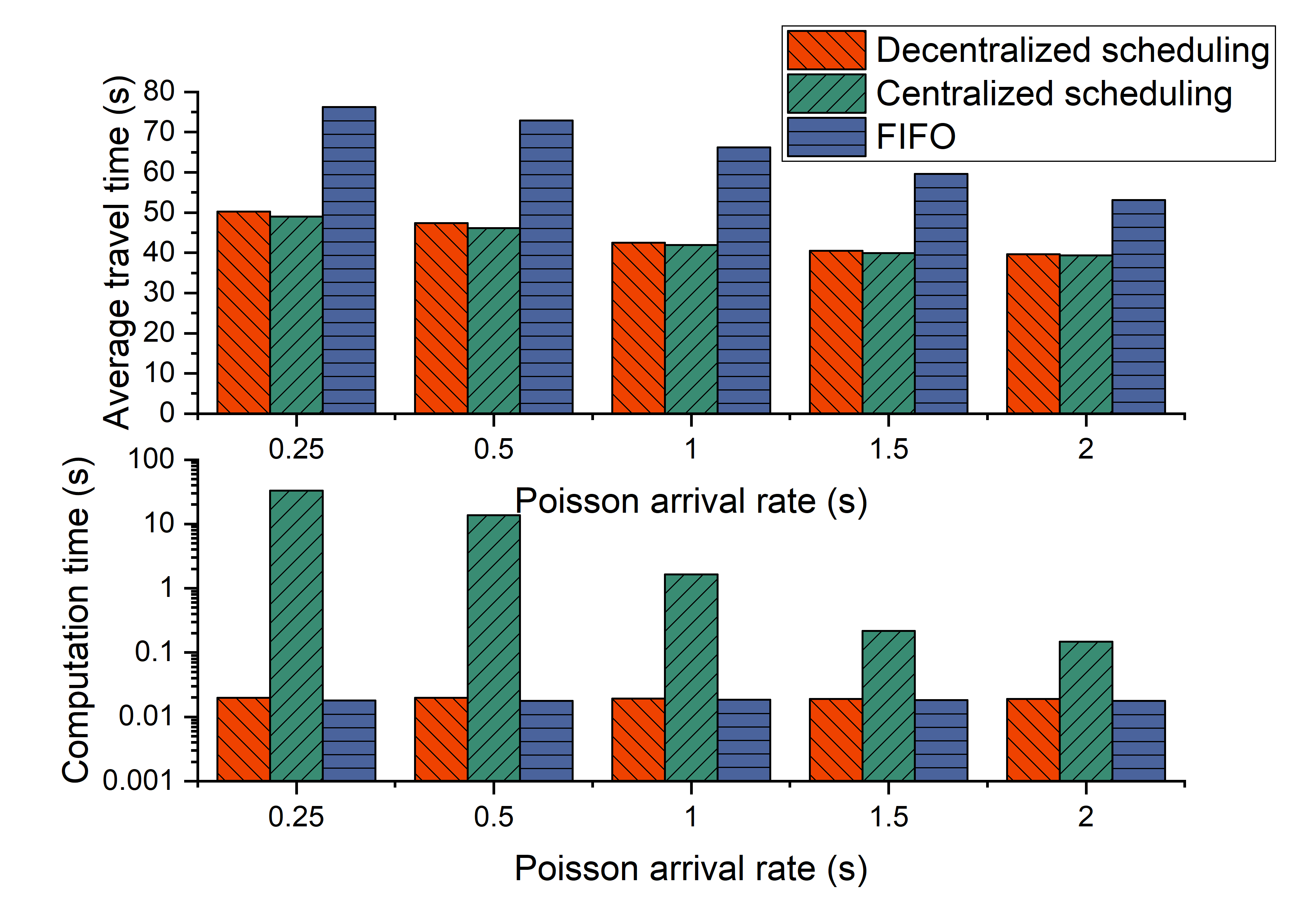}
\caption{Average travel time and computation time of decentralized scheduling, centralized scheduling, and FIFO queuing policy for $30$ vehicles with different Poisson arrival rates.}
\label{fig:comparisonUpperLevelDifferentArrival}%
\end{figure}

Since $v_\text{merge}$ is imposed a priori, upon entering the control zone CAV $i$ might find the scheduling problem (Problem \ref{Problem1}) infeasible. In this case, CAV $i$ searches for the largest speed less than $v_\text{merge}$, in which the scheduling problem has a feasible solution. Note that due to the cheap computation cost of solving the scheduling problem (Table \ref{tbl:computation}), this can be achieved in real time. However, if there is a scenario of high-traffic volume and $v_{\min} > 0$, the scheduling problem may still have not a feasible solution. To address this problem, it is required to set $v_{\min}=0$, implying that the latest feasible time for CAV $i$ to travel at zone $m\in\mathcal{I}_i$ is infinity. Thus, to satisfy the safety constraint, the arrival time at zone $m$ might be too big resulting in CAV $i$ reaches to a full stop if the road reaches its maximum capacity. 

\begin{table}[htbp]
\caption{The mean, and standard deviation of computation times of CAVs in solving the MILP.}
\vspace{0.5em}
\centering
\begin{tabular}{c|c|c|c} \label{tbl:computation}
    Traffic volume & Total vehicles& Mean & Standard deviation 
    \\
    (veh/h)   & &         (ms)       & (ms)          
    \\
    \toprule
        $400$  &16&  $21.0$  &   $2.9$\\
        $600$  &19&  $21.3$  &   $1.6$\\
        $800$  &25&  $23.8$  &   $2.3$\\
        $1000$ &29&  $24.1$   &  $2.2$\\
        $1200$ &35&  $25.4$  &   $1.9$\\
\end{tabular}
\end{table}

\begin{figure*}[ht]
    \centering
\subfloat[][]{\includegraphics[width=0.33\linewidth]{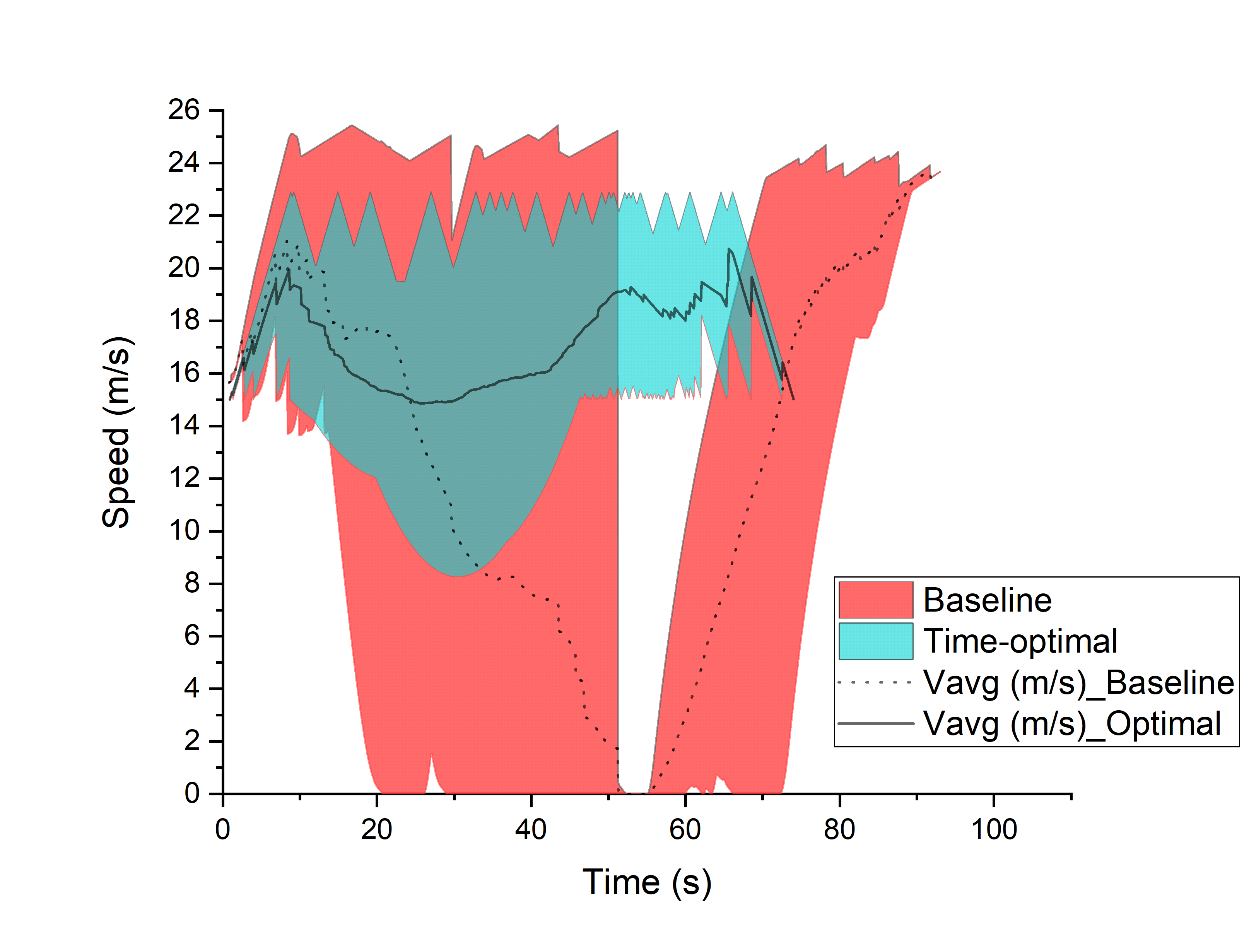}\label{a}}
\subfloat[][]{\includegraphics[width=0.33\linewidth]{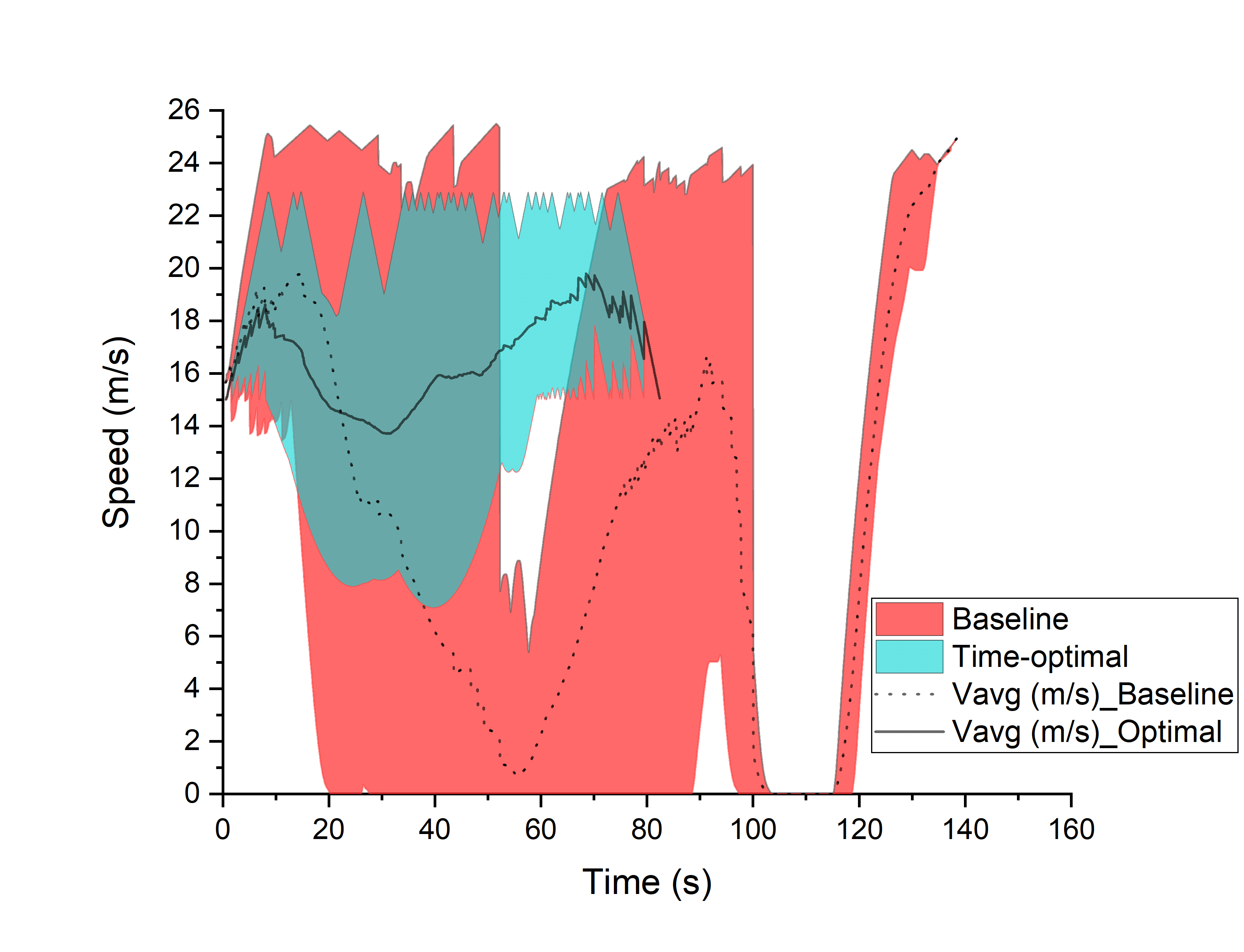}\label{b}}
\subfloat[][]{\includegraphics[width=0.33\linewidth]{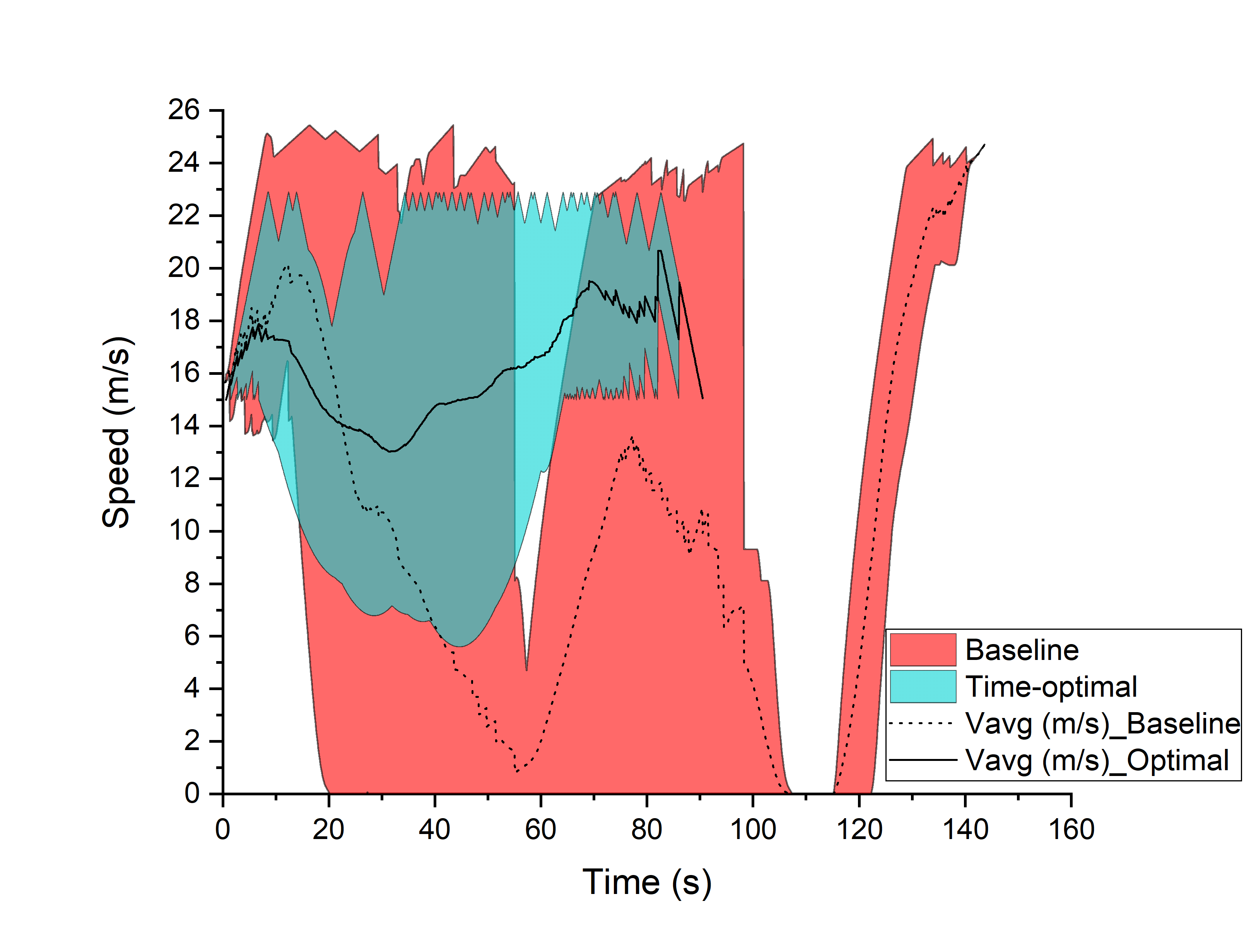}\label{c}}
    \caption{The instantaneous average, maximum and minimum speed of CAVs inside the control zone for the baseline and optimal scenarios with traffic volume \protect\subref{a} 600 veh/h, \protect\subref{b} 1000 veh/h and \protect\subref{c} 1200 veh/h.}
    \label{fig:rangeSpeed}
\end{figure*}

\begin{figure}
\centering
\includegraphics[width=0.9\linewidth]{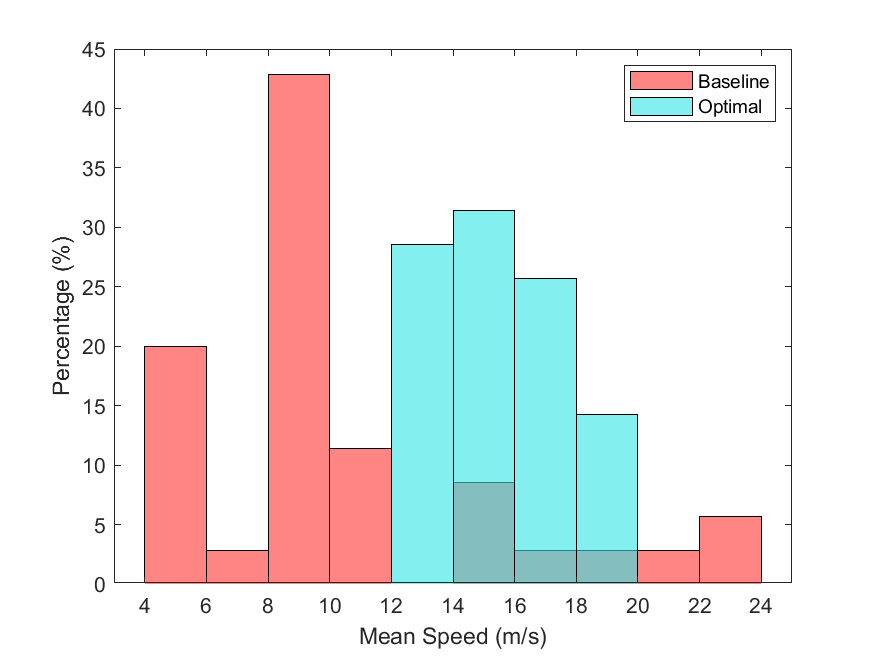}
\caption{A relative frequency histogram for mean speed of each vehicle for the baseline and optimal scenarios with traffic volume 1200 veh/h.}
\label{fig:histMeanSpeed}%
\end{figure}

\section{Concluding Remarks and Future Research}\label{sec:conc}

In this paper, we proposed a decentralized time-optimal control framework for CAVs at adjacent intersections.
We established a hierarchical optimal control framework for the coordination of CAVs consisting of three upper-level problems and one low-level problem. In the upper level, we formulated a scheduling problem that each CAV solves upon entering the control zone. The upper-level problem's outcome becomes the input of the low-level problem, which is the tuple of optimal arrival time at each zone to avoid the lateral and rear-end collision and minimize the CAV's travel time.  
In the low-level control, we formulated an optimal control problem, the solution of which yields the optimal control input (acceleration/deceleration) minimizes the transient engine operation. We derived an analytical solution for each zone that can be implemented in real time. Finally, we demonstrated the effectiveness of the proposed framework through simulation, and compared the results with the baseline scenario in different traffic volumes.

As we discussed in Section \Rmnum{3}, solving a constrained solution leads to solving a system of non-linear equations that might be hard to solve in real-time for some cases. However, a different approach to address this problem has been discussed in \cite{Malikopoulos2019CDC} and \cite{Malikopoulos2020}, in which the upper-level optimization problem yield a time that results in the unconstrained energy-optimal solution in the low-level problem. This approach has also been validated in \cite{chalaki2020experimental} at University of Delaware's Scaled Smart City for multi-lane roundabouts. Several research efforts considered a lane choice as an integer variable in their MILP formulation \cite{yu2019corridor,yu2018integrated,wong2003lane}; however, increasing the number of integer variables usually increases the computational complexity of the problem, which renders them inapplicable to real-time applications. Although in this paper we proposed an approach for two adjacent intersections, it is of interest to explore the extension of these methods to a transportation network, including multiple traffic scenarios, such as roundabouts, intersections, merging roadways, and left turn bay without increasing the computational complexity. For instance, in \cite{chalaki2020TCST}, we proposed a recursive algorithm for the upper-level planning for multiple intersections, the solution of which yields the energy-optimal arrival times at each intersection and the lane choice.

Ongoing research considers the presence of noise and error in the framework originated from the vehicle-level control and also investigates the effects of errors and delays in the vehicle-to-vehicle and vehicle-to-infrastructure communication. Another potential direction for future research is to consider coordination for mixed-traffic scenarios and the interaction of human-driven vehicles and CAVs. To consider mixed-traffic scenarios in our framework, we need to have a unit for trajectory prediction of human-driven vehicles. Considering human-driven vehicles to have a fixed schedule, each CAV attempts to formulate the scheduling problem and solve it. However, to ensure safety throughout the control zone, we may need to have a dynamic scheduling in which the scheduling problem may need to be re-solved based on a new prediction when the actual trajectory of human-driven vehicles deviates from the old prediction.

\begin{appendices}

\section{Proof of Lemma \ref{Lemma-oneSP-U}}\label{AA}
If the state constraints for the release time problem (Problem \ref{problem2}) are not active, this implies that $\mu_i^c = \mu_i^d=0$. Solving \eqref{24} and \eqref{25} we have $\lambda_i^{p{^\ast}}(t)=a_i$ and $\lambda_i^{v{^\ast}}(t)=-a_it+b_i$, where $a_i$, $b_i$ are the constants of integration.
From Pontryagin's minimum principle, the optimal control input should satisfy the following condition

\begin{gather}
    H(t,p_i^\ast(t),v_i^\ast(t),u_i^\ast(t),\lambda_i^\ast(t))\leq H(t,p_i^\ast(t),v_i^\ast(t),u_i(t),\lambda_i^\ast(t)).
\end{gather}
Substituting \eqref{hamil1} in above equation, and simplifying yields 
\begin{gather}
\lambda_i^{v{^\ast}}(t)~u_i^\ast(t)\leq \lambda_i^{v{^\ast}}(t)~u_i(t),
\end{gather}

Therefore, $u^\ast(t)$ is found as follows  
\begin{gather}
 u_i^\ast(t)=\left\{ \begin{array}{ll}
u_{\min},& \mbox{if}\quad \lambda_i^{v{^\ast}}(t)>0\\
u_{\max} ,& \mbox{if}\quad \lambda_i^{v{^\ast}}(t)<0\\
           \end{array}. \right. 
\end{gather}
It follows immediately from the linearity of  $\lambda_i^{v{^\ast}}(t)$ that its sign can change at most once.
For the second statement, we use the fact that we can have at most one switching point. There are four cases that we should consider: Case 1: $u_i^\ast(t) =u_{\min}$, Case 2: $u_i^\ast(t)=u_{\max}$, Case 3: $u_i^\ast(t)=u_{\max}$ and then it switches to $u_i^\ast(t)=u_{\min}$, and Case 4: $u_i^\ast(t)=u_{\min}$ and then it switches to $u_i^\ast(t)=u_{\max}$. 
The initial and final states, denoted by $[ p^s_i,v^s_i]^\top$ and $[p^e_i,v^e_i]^\top$ respectively, are known. 

Case 1: If CAV $i\in\mathcal{N}(t)$ decelerates with $u_{\min}$, then from (\ref{27a}) its final speed is
\begin{gather}\label{finalvel2}
 v^{f}_i=\sqrt[]{2u_{\min}\cdot(p^e_i-p^s_i)+{v^s_i}^2}.
\end{gather}

Case 2: If CAV $i\in\mathcal{N}(t)$ accelerates with $u_{{\max}}$, similarly from (\ref{27a}) its final speed is
\begin{gather}\label{finalvel2}
 v^{f}_i=\sqrt[]{2u_{{\max}}\cdot(p^e_i-p^s_i)+{v^s_i}^2}.
\end{gather}

Case 3: We have $u_{{\max}}$ then $u_{\min}$, this implies the following:
\begin{equation}
\lambda_i^{v{^\ast}}(t)=\left\{ \begin{array}{ll}
-&,\: \mbox{if}\quad t_i^{s,m}\leq t<t^{c,m}_i\\
0 &,\: \mbox{if}\quad t =t^{c,m}_i \\
+ &,\: \mbox{if}\quad t^{c,m}_i < t\leq t_i^{e,m}\\
           \end{array}, \right. 
\end{equation}
where $t^{c,m}_i$ is the time that the control input changes sign, and $\dot{\lambda_i^{v{^\ast}}}(t)=-\lambda_i^{p{^\ast}} = -a_i >0$. Evaluating the Hamiltonian along the optimal control at $t^{c,m}_i$ yields 
\begin{align}\label{hamil-tc}
&H_i(t^{c,m}_i,p_i^\ast(t^{c,m}_i),v_i^\ast(t^{c,m}_i),u_i^\ast(t^{c,m}_i),\lambda_i^{\ast}(t^{c,m}_i))\nonumber\\&=1+\lambda_i^{p{^\ast}} v_i^\ast (t^{c,m}_i).
\end{align}
Since the final time $t = t^{e,m}$ is not specified, the transversality condition gives
\begin{equation}\label{hamil-tf}
    H_i(t^{e,m}_i,p_i^\ast(t^{e,m}_i),v_i^\ast(t^{e,m}_i),u_i^\ast(t^{e,m}_i),\lambda_i^{\ast}(t^{e,m}_i))=0.
\end{equation}
Additionally, the Hamiltonian \eqref{hamil1} must be constant along the optimal solution, since it is not an explicit function of time
\begin{equation}\label{hamil-tf=tc}
    1+\lambda_i^{p{^\ast}} v_i^\ast (t^{c,m}_i)=0.
\end{equation}
Hence, $v_i^\ast (t^{c,m}_i) = -\dfrac{1}{\lambda_i^{p{^\ast}}}>0$. 

Case 4 : Similarly to Case 3, it can be shown that $\lambda_i^{p{^\ast}}=a_i>0$ . Solving \eqref{hamil-tf=tc} for $v_i^\ast (t^{c,m}_i)$, we have $v_i^\ast (t^{c,m}_i) = -\dfrac{1}{\lambda_i^{p{^\ast}}}<0$.
Hence, this case cannot be a feasible solution.\qed

\section{Proof of Theorem \ref{theorem-endtime-intermediateTime-con}}\label{B}

Let CAV $i\in \mathcal{N}(t)$ enter and exit the zone $m\in\mathcal{I}_i$ at $t_i^s$ and $t_i^e$ respectively. From the boundary conditions, we have $p_i(t_i^s) =p_i^s$, $v_i(t_i^s) =v_i^s$, $p_i(t_i^e) =p_i^e$ and $v_i(t_i^e) =v_i^e$. 
CAV $i$ cruises with $u_i(t)=u_{{\max}}$, and then it enters the constrained arc $v_i(t) = v_{\max}$ at time $\tau_1$. It stays at the constrained arc with $u_i(t)= 0$ until time $\tau_2$. After exiting the constrained arc, it decelerates with $u_i(t)=u_{{\min}}$. Substituting the optimal control input in \eqref{27a} yields the following optimal state equations: 
\begin{align}
p_i^{\ast}(t)=&\frac{1}{2}u_{{\max}}(t^2-{t_i^s}^2)-u_{{\max}}t_i^s(t-t_i^s)+v_i^s(t-t_i^s)+p_i^s, \nonumber \\
v_i^{\ast}(t)=&u_{{\max}}(t-t_i^s)+v_i^s ,\quad \forall~ t\in [t_i^s,\tau_1^-].\\
p_i^{\ast}(t)=& v_{{\max}}(t-\tau_1)+p_i^{\ast}(\tau_1^+),\nonumber
\\
v_i^{\ast}(t)=&v_{{\max}}, \quad \forall~ t\in [\tau_1^+,\tau_2^-].\\
p_i^{\ast}(t)=&\frac{1}{2}u_{{\min}}(t^2-\tau_2^2)-u_{{\min}}\tau_2(t-\tau_2)\nonumber \\
&+v_i^{\ast}(\tau_2^+)(t-\tau_2)+p_i^{\ast}(\tau_2^+), \nonumber
\\
v_i^{\ast}(t)=&u_{{\min}}(t-\tau_2)+v_i^{\ast}(\tau_2^+),\quad \forall~ t\in [\tau_2^+,t_i^e].
\end{align}
The states of CAV are continuous at $\tau_1$ and $\tau_2$, thus
\begin{align}
    p_i^{\ast}(\tau_1^-) &= p_i^{\ast}(\tau_1^+) ,
    &v_i^{\ast}(\tau_1^-) &= v_i^{\ast}(\tau_1^+),\label{cons-conditions2}\\
    p_i^{\ast}(\tau_2^-) &= p_i^{\ast}(\tau_2^+),
    &v_i^{\ast}(\tau_2^-) &= v_i^{\ast}(\tau_2^+)\label{cons-conditions3}.
\end{align}
From \eqref{cons-conditions2}-\eqref{cons-conditions3} and the boundary conditions, piecing the unconstrained and constrained arcs together, we have
\begin{align}
    &p_i^{\ast}(\tau_1)=\frac{1}{2}u_{{\max}}(\tau_1^2-t_i^s)-u_{{\max}}t_i^s(\tau_1-t_i^s) \nonumber\\
    &+v_i^s(\tau_1-t_i^s)+p_i^s,\\
    &\tau_1 = \frac{v_{{\max}}-v_i^s}{u_{{\max}}}+t_i^s ,\\
    &p_i^{\ast}(\tau_2)=v_{{\max}}(\tau_2-\tau_1)+p_i^{\ast}(\tau_1),\\
    &v_i^{\ast}(\tau_2) = v_{{\max}},\\
    &p_i^e=\frac{1}{2}u_{{\min}}({t_i^e}^2-\tau_2^2)-u_{{\min}}\tau_2({t_i^e}-\tau_2),\nonumber \\
&+v_i^{\ast}(\tau_2)({t_i^e}-\tau_2)+p_i^{\ast}(\tau_2),\\
&v_i^e=u_{{\min}}({t_i^e}-\tau_2)+v_i^{\ast}(\tau_2).
\end{align}
Solving the system of equations above yields \eqref{processtimeEqCons}. \qed

\section{Proof of Lemma \ref{lemma-notactivatingRearEnd}}\label{D}       

Since CAV $i$ and $k$ cruise on the same lane, \eqref{Schedulecons2} simplifies to $T_i^{^{m}}\geq T_k^{^{m}} + h$. For CAV $k$ we have
\begin{align}
        &p_k(t) < p_k(t^\prime), \quad \forall~ t<t^\prime \in\mathbb{R}^{+},\\
&\inf(p_k(T_i^{^{m}})) = p_k(T_k^{^{m}} + h)   ,\quad  T_i^{^{m}}\geq T_k^{^{m}} + h.
\end{align}
Evaluating the \eqref{RearEndCons} at time $T_i^{^{m}}$ yields 
\begin{equation}\label{rearendsafety-seteval}
    (p_k(T_i^{^{m}})-p_k(T_k^{^{m}}))-(p_i(T_i^{^{m}})-p_i(T_i^{^{m}})) \geq \gamma_i + \varphi v_i(T_i^{^{m}}).
\end{equation}
Then, we have
\begin{equation}\label{rearendSet1eq}
    (p_k(T_i^{^{m}})-p_k(T_k^{^{m}})) \geq \gamma_i + \varphi v_i(T_i^{^{m}}),
\end{equation}
\begin{equation}
    (p_k(T_i^{^{m}})-p_k(T_k^{^{m}})) \geq \inf( (p_k(T_i^{^{m}})-p_k(T_k^{^{m}}))),
\end{equation}
where
\begin{equation}
    \inf( (p_k(T_i^{^{m}})-p_k(T_k^{^{m}}))) = p_k(T_k^{^{m}} + h) - p_k(T_k^{^{m}}).
\end{equation}

If 
\begin{equation}\label{rearendSet2eq}
    p_k(T_k^{^{m}} + h)-p_k(T_k^{^{m}}) > \gamma_i + \varphi v_i(T_i^{^{m}})
\end{equation}
holds, then \eqref{rearendSet1eq} also holds, and the rear-end safety constraint never becomes active at $t=T_i^{^{m}}$. 

The LHS of \eqref{rearendSet2eq} corresponds to the distance that CAV $k$ travelled after $h$ seconds from its entry in the zone $m$, denoted by $\Delta_k^{m}(h,u_i(t))$. Thus, 
\begin{equation}\label{argmin}
    \operatorname*{argmin}_{u_i(t)} \Delta_k^{m}(h,u_i(t))= u_{\min}.
\end{equation}
Substituting \eqref{argmin} and $t = h$ into \eqref{27a}, we have 
\begin{equation}\label{deltamin}
    \Delta_k^m = \frac{1}{2}u_{\min}h^2+v_i(T_k^{^{m}})h.
\end{equation}
If \eqref{deltamin} is greater than $\gamma_i + \varphi v_i(T_i^{^{m}})$, it yields \eqref{rearendSet2eq}, and the proof is complete. \qed


\end{appendices}

\bibliographystyle{IEEEtran} 
\bibliography{bib/IDS_Publications_10152021.bib, bib/ref.bib}


\begin{IEEEbiography}
	[{\includegraphics[width=1.1in,height=1.25in,clip,keepaspectratio]{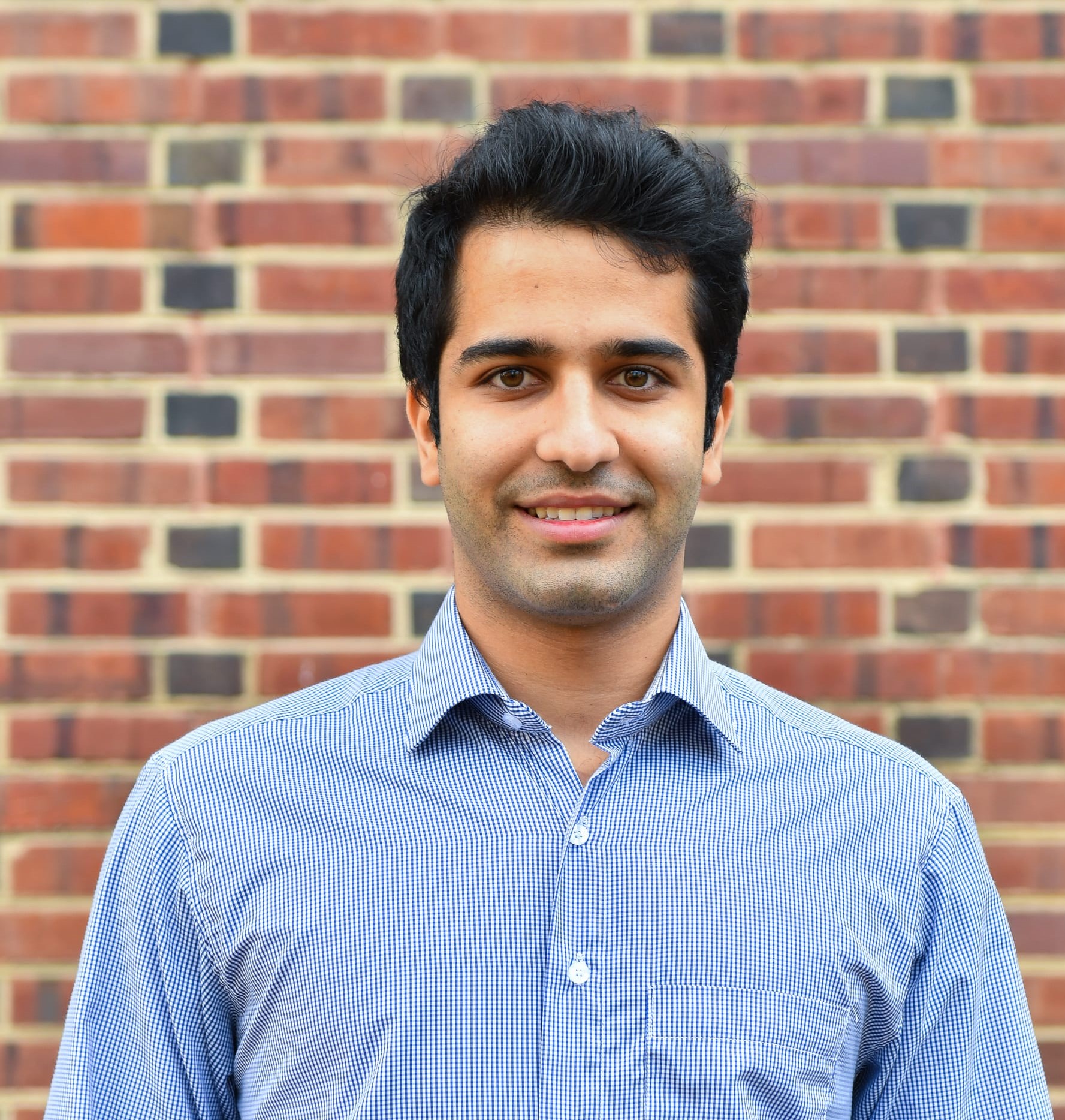}}]
	{Behdad Chalaki} (S’17) received the B.S. degree in mechanical engineering from the University of Tehran, Iran in 2017. He is currently a Ph.D. student in the Information and Decision Science Laboratory in the Department of mechanical engineering at the University of Delaware. His primary research interests lie at the intersections of decentralized optimal control, statistics, and machine learning, with an emphasis on transportation networks. In particular, he is motivated by problems related to improving traffic efficiency and safety in smart cities using optimization techniques. He is a student member of IEEE, SIAM, ASME, and AAAS.
\end{IEEEbiography}
\begin{IEEEbiography}[{\includegraphics[width=1.1in,height=1.25in,clip,keepaspectratio]{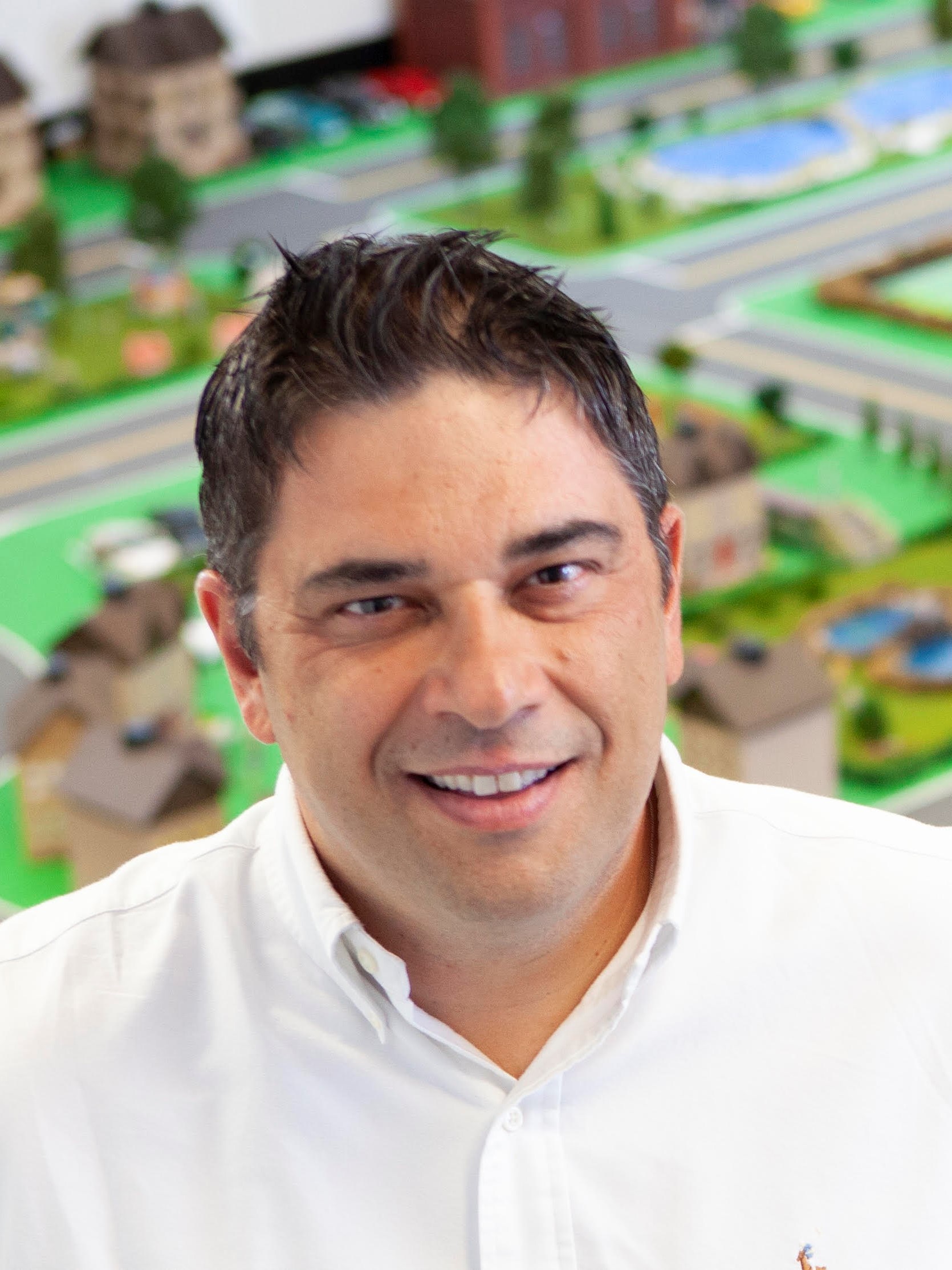}}]{Andreas A. Malikopoulos}
	(S’06–M’09–SM’17)  received the Diploma in mechanical engineering from the National Technical University of Athens, Greece, in 2000. He received M.S. and Ph.D. degrees from the department of mechanical engineering at the University of Michigan, Ann Arbor, Michigan, USA, in 2004 and 2008, respectively. 
	He is the Terri Connor Kelly and John Kelly Career Development Associate Professor in the Department of Mechanical Engineering at the University of Delaware (UD), the Director of the Information and Decision Science (IDS) Laboratory, and the Director of the Sociotechnical Systems Center. Before he joined UD, he was the Deputy Director and the Lead of the Sustainable Mobility Theme of the Urban Dynamics Institute at Oak Ridge National Laboratory, and a Senior Researcher with General Motors Global Research \& Development. His research spans several fields, including analysis, optimization, and control of cyber-physical systems; decentralized systems; stochastic scheduling and resource allocation problems; and learning in complex systems. The emphasis is on applications related to smart cities, emerging mobility systems, and sociotechnical systems. He has been an Associate Editor of the IEEE Transactions on Intelligent Vehicles and IEEE Transactions on Intelligent Transportation Systems from 2017 through 2020. He is currently an Associate Editor of Automatica and IEEE Transactions on Automatic Control. He is a member of SIAM, AAAS, and a Fellow of the ASME.
\end{IEEEbiography}

\end{document}